\documentclass[12pt,a4paper]{amsart}
\usepackage{version}
\usepackage{amssymb}
\usepackage{amsfonts}

 \numberwithin{equation}{section}
\theoremstyle{plain}
\newtheorem{theorem}{Theorem}[section]
\newtheorem{lemma}[theorem]{Lemma}
\newtheorem{proposition}[theorem]{Proposition}
\newtheorem{corollary}[theorem]{Corollary}
\newtheorem{example}[theorem]{Example}
\newtheorem{definition}[theorem]{Definition}
\newtheorem{remark}[theorem]{Remark}

\def\B{{\mathcal B}}

\def\E{{\mathbb E}}
\def\N{{\mathbb N}}

\def\R {{\mathbb R}}
\def\Z{{\mathbb Z}}

\def\D{{\bf D}}
\def\x{{\bf x}}

\def\P{{\mathbb P}}

\def\L{{\mathcal L}}
\def\e{{\varepsilon}}
\def\CP{{\mathcal P}}
\def\CS{{\mathcal S}}
\def\CQ{{\mathcal Q}}
\def\bS{{\bf S}}
\def\bA{{\bf A}}
\def\bE{{\bf E}}
\def\bF{{\bf F}}

\DeclareMathOperator{\diam}{diam}
\DeclareMathOperator{\dimH}{dim_H}
\DeclareMathOperator{\dimp}{dim_P}

\DeclareMathOperator{\dimbu}{\overline{dim}_B}
\DeclareMathOperator{\dist}{dist}

\usepackage{color}

\begin{document}
\baselineskip 15pt

\bigskip
\baselineskip 17pt
\title[]{Dimensions of random covering sets in Riemann manifolds}

\date{}
\author[D.-J. Feng]{De-Jun Feng$^1$}
\address{Department of Mathematics, Lady Shaw Building, The Chinese University
of Hong Kong, Shatin, N. T., Hong Kong$^1$}
\email{djfeng@math.cuhk.edu.hk$^1$}
\author[E. J\"arvenp\"a\"a]{Esa J\"arvenp\"a\"a$^2$}
\address{Department of Mathematical Sciences, P.O. Box 3000,
         90014 University of Oulu, Finland$^{2,3,4}$}
\email{esa.jarvenpaa@oulu.fi$^2$}

\author[M. J\"arvenp\"a\"a]{Maarit J\"arvenp\"a\"a$^3$}
\email{maarit.jarvenpaa@oulu.fi$^3$}

\author[V. Suomala]{Ville Suomala$^4$}
\email{ville.suomala@oulu.fi$^4$}


\thanks{We acknowledge the support of RGC grants in the Hong Kong Special
Administrative Region, China (projects CUHK401112, CUHK401013), and the support
of the Centre of Excellence in Analysis and Dynamics Research funded by the
Academy of Finland. We thank Henna Koivusalo and Antti K\"aenm\"aki for
interesting discussions.}

\begin{abstract} Let ${\pmb M}$, ${\pmb N}$ and ${\pmb K}$  be $d$-dimensional
Riemann manifolds. Assume that $\bA:=(A_n)_{n\in\N}$ is a sequence of Lebesgue
measurable subsets of ${\pmb M}$ satisfying a necessary density condition and
$\x:=(x_n)_{n\in\N}$ is a sequence of independent random variables which are
distributed on ${\pmb K}$ according to a measure which is not purely singular
with respect to the Riemann volume. We give a formula for the almost sure value
of the Hausdorff dimension of random covering sets
$\bE(\x,\bA):=\limsup_{n\to\infty}A_n(x_n)\subset {\pmb N}$. Here $A_n(x_n)$ is a
diffeomorphic image of $A_n$ depending on $x_n$. We also verify that the
packing dimensions of $\bE(\x,\bA)$ equal $d$ almost surely.
\end{abstract}

\maketitle

\section{Introduction and main theorem}\label{intro}

Limsup sets, defined as upper limits of various sequences of sets, play
an important role in different areas of mathematics.
One of the earliest occurrences originates from the study of
random placement of circular arcs in the unit circle by Borel {\cite{Bor1897}} in the late
1890's. He stated that a given point belongs to infinitely many arcs
provided that the placement of arcs is random and the sum of
their lengths is infinite. This statement is the origin of what is
nowadays known as the Borel-Cantelli lemma. We refer to
\cite{Ka00} for more details and references on the historical development.
Related to geometric measure theory and fractals, limsup sets appear
implicitly already in the investigation of the Besicovitch-Eggleston sets
concerning the $k$-adic expansions of real numbers {\cite{Be34, Eg49}}. They play also a central
role in Diophantine approximation. For instance, the classical theorems of
Khintchine and Jarnik provide size estimates in terms of Lebesgue and
Hausdorff measure for limsup sets consisting of well-approximable numbers {(cf. \cite{Ha98})}.

In the modern language, random covering sets are a class of limsup sets
defined by means of a family of randomly distributed
subsets of the $d$-dimensional torus $\mathbb T^d:=\mathbb R^d/\mathbb Z^d$.
Supposing that $\bA:=(A_n)_{n\in\N}$ is a deterministic sequence of
non-empty subsets of $\mathbb T^d$ and $\x:=(x_n)_{n\in\N}$ is
a sequence of independent random variables which are uniformly distributed
on $\mathbb T^d$, define a random covering set $\bE(\x,\bA)$ by
\begin{equation*}
\bE(\x,\bA):=\limsup_{n\to\infty}(x_n+A_n)
  =\bigcap_{n=1}^\infty\bigcup_{k=n}^\infty(x_k+A_k),
\end{equation*}
where $x+A:=\{x+y:y\in A\}$. Denoting the Lebesgue measure on
$\mathbb T^d$ by $\mathcal L$, it follows from Borel-Cantelli
lemma and Fubini's theorem that almost surely either
$\mathcal L(\bE(\x,\bA))=0$
or $\mathcal L(\bE(\x,\bA))=1$ depending on whether the series
$\sum_{k=1}^\infty \mathcal L(A_k)$ converges or diverges, respectively.
Note that this result is essentially the higher
dimensional analogue of Borel's original idea concerning the covering of
the circle by random arcs which we discussed in the beginning of this
section.

The case of full Lebesgue measure has been extensively studied.
In 1956 Dvoretzky ~\cite{Dv} posed a problem of finding conditions which
guarantee that the whole torus $\mathbb T^d$ is covered almost surely.
Even in the simplest case when $d=1$ and the generating sets are intervals
of length $(l_n)_{n\in\N}$, this problem, known in literature as the Dvoretzky
covering problem, turned out to be rather long-standing. After substantial
contributions of many authors, including Billard ~\cite{Bi},
Erd{\H{o}s} ~\cite{Er}, Kahane ~\cite{Ka} and
Mandelbrot ~\cite{Man}, the full answer was
given in this context by Shepp ~\cite{Sh} in 1972. He verified that
$\bE(\x,\bA)=\mathbb T^1$ almost surely if and only if
\[
\sum_{n=1}^\infty\frac {1}{n^2}\exp(l_1+\dots+l_n)=\infty,
\]
where the lengths $(l_n)_{n\in\mathbb N}$ are in decreasing order. In full
generality, the Dvoretzky covering problem is still unsolved. The higher
dimensional case has been studied by El H\'elou ~\cite{ElHel} and
Kahane ~\cite{Ka90} among others. In ~\cite{Ka90}, a complete solution is
provided in the case when generating sets are similar simplexes.

For various other aspects of random covering sets, we refer to \cite{BF,
ElHel, Fa, FK, Ha, Ja, Ka85, Ka90, Ka00, LiShXi, Man2, Sh2}.
Recent contributions to the topic include various types of
dynamical models, see ~\cite{FST, JoSt, LiSe}, and
projectional properties ~\cite{CKLS}.

Further motivation to study limsup sets stems from Diophantine approximation.
Recall that for $\phi\colon\mathbb N\to]0,\infty[$, the set of $\phi$
well-approximable numbers consists of those $x\in\mathbb R$ for which
\begin{equation*}
\left|x-\frac{p(q)}q\right|<\phi(q)\text{ for infinitely many }q\in\mathbb N.
\end{equation*}
Given $\phi$, the determination of the size of these limsup sets and various
variants is an important theme in Diophantine approximation and there is
a vastly growing literature on this branch of metric number theory,
see {\cite{BDV, BV} and the references therein}.

In the circle $\mathbb T^1$, the study of
$\phi$ well-approximable numbers may be regarded as a special
case of the shrinking target problem or
dynamical Diophantine approximation formulated in the following manner:
assuming that $X$ is a metric space, $T:X\to X$ is a dynamical system,
$(r_n)_{n\in\mathbb N}$ is a sequence of positive real numbers and
$x_0\in X$, determine the size of the set
\[
\limsup_{n\to\infty}B(T^n(x_0),r_n)=\{x\in X:T^n(x)\in B(x_0,r_n)
 \text{ for infinitely many }n\in\N\},
\]
where $B(x,r)$ is the open ball with radius $r$ centred at $x\in X$. Indeed,
letting $x_0=0$, $r_q=q \phi(q)$ and $T\colon\mathbb T^1\to\mathbb T^1$ be the
rotation by an angle $x$, we recover the case of $\phi$ well-approximable
numbers. A variant of this question, called the moving target problem, is
concerned with the investigation of the limsup set
\[
\{x\in X\mid x\in B(T^n(x_0),r_n)\text{ for infinitely many }n\in\N\},
\]
see \cite{BD, Bu03}. A recent account on this line of research is
provided in \cite{FST}. For an interesting application of limsup sets to the
study of Brownian motion, we refer to \cite{KPX}.

In this paper, we focus on the natural problem of determining almost sure
values of Hausdorff and packing dimensions of random covering
sets in the case when they have zero Lebesgue measure. We denote the
Hausdorff and packing dimensions by $\dimH$ and $\dimp$, respectively.
For $d=1$ and for an arbitrary decreasing sequence $\bA=(A_n)_{n\in\N}$ of
intervals of lengths $(l_n)_{n\in\N}$, the almost sure Hausdorff dimension of the
random covering set is given by
\begin{equation}\label{Duformula}
\dimH\bE(\x,\bA)=\inf\bigl\{t\geq 0:\sum_{n=1}^\infty(l_n)^t<\infty\bigr\}
=\limsup_{n\to\infty}\frac{\log n}{-\log l_n}.
\end{equation}
For $l_n=n^{-\alpha}$, $\alpha>1$, this is proved by Fan and Wu \cite{FW} and, as
explained in their paper, the method works also for more general decreasing
sequences $(l_n)_{n\in\N}$. Using  an approach  different from that of \cite{FW},
Durand \cite{Du} generalised the result of \cite{FW} and obtained a dichotomy
result for the Hausdorff measure of $\bE(\x,\bA)$ for general gauge functions.
The dimension result \eqref{Duformula}, as well as its analogy in
$\mathbb T^d$ for random coverings with balls, can also be derived from
the mass transference principle proved by Beresnevich and Velani in \cite{BV},
see \cite{JJKLS}. In addition to random covering sets, the mass transference
technique has proved to be a useful tool in studying the limsup sets in the
context of Diophantine approximation and shrinking target problems. See e.g.
\cite{BDV, BV, FST, HV}. However, its applicability is
essentially limited to the case when the sequence $\bA$ consists of balls
and, therefore, it cannot be utilised in the general setting of this paper.

The methods used in \cite{Du, FW} rely essentially on
the ambient space being a torus and generating sets being balls.
The question of calculating  the dimensions of random
covering sets in the $d$-dimensional torus was first addressed in \cite{JJKLS}
in the case when the generating sets are  rectangle-like. More precisely, assume that the generating sets in $\bA$  are  of the form  $A_n=\Pi(L_n(R))$ for all
$n\in\N$, where $\Pi:\mathbb R^d\to\mathbb T^d$ is a natural covering map, $R$
is a subset of the closed unit cube $[0,1]^d$ with non-empty interior and, for
all $n\in\N$, the map $L_n:\mathbb R^d\to\mathbb R^d$ is a contracting linear
injection such that the sequences of singular values of $(L_n)_{n\in\N}$
decrease to $0$ as $n$ tends to infinity. Note that the singular values of
$L_n$ are the lengths of the semi-axes of $L_n(B(0,1))$. Under this assumption, according to
\cite{JJKLS}, almost surely the Hausdorff dimension of $\bE(\x,\bA)$ is
given in terms of singular value functions $\Phi^t(L_n)$ (for the definition
see \cite{JJKLS}), that is, almost surely
\begin{equation}\label{JJKLSformula}
\dimH\bE(\x,\bA)=\inf\bigl\{0<t\le d:\sum_{n=1}^\infty\Phi^t(L_n)<\infty\bigr\}
\end{equation}
with the interpretation $\inf\emptyset=d$ (see \cite{JJKLS}).

In \cite{Pe}, Persson  proved that \eqref{JJKLSformula} remains valid when dropping off the
monotonicity assumption on the singular values of $(L_n)_{n\in\N}$ in  \cite{JJKLS}. Indeed, he
showed that for a sequence $\bA$ of open subsets of $\mathbb T^d$, almost
surely
\begin{equation}\label{Persson}
\dimH\bE(\x,\bA)\ge\inf\bigl\{0<t\le d:\sum_{n=1}^\infty g_t(A_n)<\infty\bigr\},
\end{equation}
where
\begin{equation}\label{gtdef}
g_t(F):=\frac{\mathcal L(F)^2}{I_t(F)}
\end{equation}
for all Lebesgue measurable sets $F\subset\mathbb T^d$ with $\L(F)>0$,  and
\begin{equation}\label{setenergy}
I_t(F):=\iint_{F\times F}\vert x-y\vert^{-t}\,d\mathcal L(x)d\mathcal L(y)
\end{equation}
is the $t$-energy of $F$. For simplicity, we use the notation $|x-y|$ for
both the Euclidean distance and the natural distance in $\mathbb T^d$.  When  the generating sets of  $\bA$
are open rectangles, the lower bound in \eqref{Persson} equals the
right-hand side of \eqref{JJKLSformula}.

As will be verified by Example \ref{G>g}, the lower bound in \eqref{Persson} is not optimal. As our first goal, we aim at an exact dimension formula for the
random covering sets constructed from an arbitrary sequence $\bA$ of Lebesgue
measurable sets satisfying a mild density condition.
To this end, we introduce the following notation.
For $0\le t<\infty$, the $t$-dimensional Hausdorff content of a set
$F\subset\mathbb R^d$ is denoted by
\begin{equation}\label{Hauscontent}
\mathcal H_\infty^t(F)
:=\inf\bigl\{\sum_{n=1}^\infty(\diam F_n)^t:F\subset\bigcup_{n=1}^\infty F_n\bigr\},
\end{equation}
where $\diam$ is the diameter of a subset of $\R^d$. For a sequence
$\bA=(A_n)_{n\in\mathbb N}$ of subsets of $\mathbb R^d$, we define
\begin{equation}\label{deft0}
t_0(\bA)
:=\inf\bigl\{0\le t\le d:\sum_{n=1}^\infty\mathcal H_\infty^t(A_n)<\infty\bigr\}
\end{equation}
with the interpretation $\inf\emptyset=d$. If $\bA$
is a sequence of Lebesgue measurable subsets of $\mathbb R^d$, set
\begin{equation}\label{defs0}
s_0(\bA):=\sup\bigl\{0\le s\le d:\sum_{n=1}^\infty G_s(A_n)=\infty\bigr\}
\end{equation}
where
\begin{equation}\label{defG}
G_s(E):=\sup\{g_s(F):F\subset E,\; F\text{ is Lebesgue measurable and }
\mathcal L(F)>0\}
\end{equation}
with the interpretation $\sup\emptyset=0$. Finally, given
$F\subset\mathbb R^d$, we say that a point $x\in F$ has
{\it positive Lebesgue density with respect to $F$} if
\[
\liminf_{r\to 0}\frac{\mathcal L(F\cap B(x,r))}{\L(B(x,r))}>0
\]
and, moreover, the set $F$ has {\it positive Lebesgue density} if all of its
points have positive Lebesgue density with respect to $F$.

As a consequence of our main theorem (see Theorem \ref{main}), we will prove
that almost surely,
\begin{equation}\label{Hausdorff torus}
\dimH\bE(\x,\bA)=s_0(\bA)=t_0(\bA)
\end{equation}
provided that $\bA=(A_n)_{n\in\mathbb N}$ is a sequence of Lebesgue measurable
subsets of $\mathbb T^d$ having positive Lebesgue density. We note that if $U$
is open, then a straightforward approximation argument implies that
\begin{equation*}
G_s(U)=\sup\{g_s(V):V\subset U,\; V\text{ is open and }\mathcal L(V)>0\}.
\end{equation*}
With Persson's result, this characterisation can be employed to give
a more direct proof of the fact that $s_0(\bA)$ is a lower bound
for $\dimH\bE(\x,\bA)$ in the case  when  $\bA$ is a sequence of open
sets. However, this method does not work if the sets in the sequence $\bA$
fail to be open. For this reason, we need to make use of
a completely different approach to deal with a more general
generating sequence $\bA$. For the purpose of proving
\eqref{Hausdorff torus} in full generality, we will introduce
the notion of minimal regular energy (see Section \ref{regularenergy}) and
utilise a technical result (Proposition \ref{thm:energy_lemma}), allowing us
to approximate a given measure $\mu$ and its $s$-energy simultaneously by
a certain sequence of normalised Lebesgue measures, as well as a sophisticated
random mass distribution argument to be carried out in
Section~\ref{lowerboundsection}.

Regarding the packing dimension of random covering sets, we
show that if the sets in $\bA$ are Lebesgue measurable and
$\mathcal L(A_n)>0$ for infinitely many $n\in\N$, then almost surely
\begin{equation}\label{packing torus}
\dimp\bE(\x,\bA)=d.
\end{equation}
For open generating sets, this result is immediate since $\bE(\x,\bA)$ is a
$G_\delta$-set, which is almost surely dense. As in the case of Hausdorff
dimension, replacing open generating sets by Lebesgue measurable ones
(of positive measure) turns out to be a subtle task. The strategy in the proof
of \eqref{packing torus} is somewhat analogous to that of
\eqref{Hausdorff torus}. However, instead of the minimal regular energy and a
mass distribution argument, we apply a result that allows us to conclude that
$\dimp\bE(\x,\bA)=d$ by estimating, for compact sets $F$, the number of dyadic
cubes intersecting $F\cap\bigcup_{i=n}^\infty (x_n+A_n)$ in a set of positive
Lebesgue measure (see Proposition \ref{pro-3.1}). Observe that since
$\bE(\x,\bA)$ is almost surely dense, the box counting dimension of
$\bE(\x,\bA)$ exists and is equal to $d$ almost surely.

To summarise, the equation \eqref{Hausdorff torus} gives a characterisation of
the almost sure value of the Hausdorff dimension of random covering sets
in $\mathbb T^d$ for rather general generating sequences $\bA$ when the
translations $\x=(x_n)_{n\in\N}$ are independent and uniformly distributed.
As illustrated by Examples
\ref{posdensityupper} and \ref{posdensity}, the assumption on positive
Lebesgue density cannot be replaced by the weaker assumption
that $\mathcal L(A_n\cap B(x,r))>0$ for all $r>0$, $x\in A_n$ and $n\in\N$.
In our main result, Theorem~\ref{main}, we will further generalise
\eqref{Hausdorff torus} and \eqref{packing torus} in several different
directions. Firstly, we will replace the uniform distribution by an arbitrary
Radon probability measure which is not purely singular with respect to the
Lebesgue measure. Secondly, we will be able to replace the torus $\mathbb T^d$
by any open subset of $\mathbb R^d$, in particular, by $\mathbb R^d$ itself.
These generalisations allow us to deduce \eqref{Hausdorff torus} and
\eqref{packing torus} for many natural unbounded models, including the case
when $(x_n)_{n\in\N}$ are independent Gaussian random variables on $\mathbb R^d$
and $(A_n)_{n\in\N}$ are Lebesgue measurable subsets with positive Lebesgue
density supported on a fixed compact subset of $\mathbb R^d$. Finally, we
extend \eqref{Hausdorff torus} and \eqref{packing torus} to Lie groups
and, more generally, to smooth Riemann manifolds. To achieve this,
note that when the ambient space is $\mathbb T^d$, the structure is linear in
the sense that the random covering set is of the form
\begin{equation}\label{linearcase}
\bE(\x,\bA)=\limsup_{n\to\infty}f(x_n,A_n)
\end{equation}
where the function $f\colon\mathbb T^d\times\mathbb T^d\to\mathbb T^d$ is
defined as $f(x,y)=x+y$. Thus, a natural attempt to extend
\eqref{Hausdorff torus} and \eqref{packing torus}
to Lie groups or, more generally, to smooth manifolds is to study
a nonlinear structure where $f$ is a smooth mapping.

Before presenting our main result in full generality, we will set up some
further notation. Let $U,V\subset\mathbb R^d$ be open sets and let
$f\colon U\times V\to\mathbb R^d$ be a $C^1$-map such that the maps
$f(\cdot,y)\colon U\to f(U,y)$ and $f(x,\cdot)\colon V\to f(x,V)$ are
diffeomorphisms for all $(x,y)\in U\times V$. Denote by $D_1f$ and $D_2f$
the derivatives of $f(\cdot,y)$ and $f(x,\cdot)$, respectively.
We assume that there exist a constant $C_u>0$ such that
\begin{equation}\label{diffeobound}
\Vert D_i f(x,y)\Vert\,,\Vert(D_i f(x,y))^{-1}\Vert \le C_u
\end{equation}
for all $(x,y)\in U\times V$ and $i=1,2$.

Let $\sigma$ be a Radon probability measure on $U$ which is not purely singular
with respect to the Lebesgue measure
$\L$. We consider the probability space $(U^\N,\mathcal F,\P)$ which is the
completion of the infinite product of $(U,\B(U),\sigma)$, where $\B(U)$ is the
Borel $\sigma$-algebra on $U$. Assuming that $\bA=(A_n)_{n\in\mathbb N}$ is a
sequence of subsets of $V$, define for all $\x\in U^\N$ a random covering
set $\bE(\x,\bA)$ by
\[
\bE(\x,\bA):=\limsup_{n\to\infty}f(x_n,A_n)
=\bigcap_{n=1}^\infty\bigcup_{k=n}^\infty f(x_k,A_k).
\]

Now we can finally present our main theorem.

\begin{theorem}\label{main}
Let $f\colon U\times V\to\R^d$ be as above and let $\Delta\subset V$ be
compact. Assume that $\bA=(A_n)_{n\in\N}$ is a sequence of non-empty subsets of
$\Delta$. Then
\begin{itemize}
\item[(a)] $\dimH\bE(\x,\,\bA)\le t_0(\bA)$ for all $\x\in U^\N$.
\item[(b)] $\dimH\bE(\x,\bA)\ge s_0(\bA)$ for $\P$-almost all $\x\in U^\N$
     provided that $\bA$ is a sequence of Lebesgue measurable sets.
\item[(c)] $\dimH\bE(\x,\bA)=s_0(\bA)=t_0(\bA)$ for $\P$-almost all $\x\in U^\N$
     provided that $\bA$ is a sequence of Lebesgue measurable sets with
     positive Lebesgue density.
\item[(d)] $\dimp\bE(\x,\bA)=d$ for $\P$-almost all $\x\in U^\N$ provided that
     $A_n$ are Lebesgue measurable and $\L(A_n)>0$ for infinitely many $n\in\N$.
\end{itemize}
\end{theorem}

It follows immediately from Theorem~\ref{main}.(d) that the upper box counting
dimension of $\bE(\x,\bA)$ equals $d$ almost surely. From the proof of
Theorem~\ref{main}.(d), we see that $\bE(\x,\bA)$ is almost surely dense in a
set of positive Lebesgue measure. Therefore, also the lower box counting
dimension equals $d$ almost surely. As a corollary of Theorem \ref{main}, we
obtain the following dimension result for random covering sets in Riemann
manifolds. Note that in Corollary~\ref{corollary}, the quantities $s_0(\bA)$
and $t_0(\bA)$ are defined as in \eqref{deft0} and \eqref{defs0}
by using the distance function induced by the Riemann metric and by
replacing $\L$ by the Riemann volume.

\begin{corollary}\label{corollary}
Let ${\pmb K}$, ${\pmb M}$ and ${\pmb N}$ be $d$-dimensional Riemann
manifolds. Assume that $f\colon{\pmb K}\times{\pmb M}\to{\pmb N}$ is a
$C^1$-map such that $f(x,\cdot)$ and $f(\cdot,y)$ are local diffeomorphisms
satisfying \eqref{diffeobound}. Let $\Delta\subset{\pmb M}$ be compact and let
$\bA=(A_n)_{n\in\N}$ be a sequence of subsets of $\Delta$. Suppose that $\sigma$
is a Radon probability measure on ${\pmb K}$ such that it is not purely
singular with respect to the Riemann volume
on ${\pmb K}$. Then the statements (a)--(d) of Theorem~\ref{main} are valid.
\end{corollary}

As mentioned earlier, choosing ${\pmb K}={\pmb M}={\pmb N}=\mathbb T^d$,
$f(x,y)=x+y$ and $\sigma=\L$,
we recover the previously mentioned setting in $\mathbb T^d$. The assumption
that the generating sets are subsets of a compact set $\Delta$ is needed,
for example, to guarantee that $\bE(\x,\bA)$ is non-empty.
A natural class of generating sets $\bA$ which satisfy the
assumptions of Theorem~\ref{main} and to which the earlier known results are
not applicable are regular Cantor sets having positive Lebesgue measure.
For the role of other assumptions in Theorem~\ref{main}, we refer to
Section~\ref{examples} where, among other things, sharpness of our results
will be discussed. Theorem~\ref{main} has a refinement concerning the
Hausdorff measures of $\bE(\x,\bA)$ with respect to doubling gauge functions.
The exact statement of this result is given in Section~\ref{finalremarks}.

The paper is organised as follows: We begin with technical auxiliary results
in Section~\ref{auxiliary}. In Section~\ref{upperbound}, we prove
Theorem~\ref{main} (a) and show that, under the assumptions of (c), we have
$s_0(\bA)=t_0(\bA)$. In Section~\ref{regularenergy}, we introduce a new concept
called minimal regular energy and show how it can be used to estimate
Hausdorff dimensions of random covering sets. Section~\ref{lowerboundsection}
is dedicated to the proof of Theorem~\ref{main} (b) whereas
the statement (d) is handled in Section~\ref{packingdimension}.
In Section~\ref{examples}, we explain how Corollary~\ref{corollary} follows
from Theorem~\ref{main} and illustrate by examples the role of the
assumptions
and the sharpness of Theorem~\ref{main}. In the last section, we give
further generalisations of  Theorem~\ref{main} and some remarks. For example, we present some results concerning  Hausdorff measures of random covering sets  with respect
to general gauge functions.

\section{Auxiliary results}\label{auxiliary}

In this section we prove technical lemmas which will be needed
in Sections \ref{upperbound}--\ref{packingdimension}.
When studying of random covering sets in the torus, one often
utilises the simple fact that $u\in x+E$ if and only if $x\in u-E$ for every
$E\subset\mathbb T^d$. In the nonlinear setting, given
a parameterised family of diffeomorphisms $W_x$, we attempt to find
a parameterised family of diffeomorphisms $X_u$ such that $u\in W_x(E)$ if and
only if $x\in X_u(E)$. It is easy to see that the linearised local version of
this problem has a solution and, therefore, this should be the case for the
original nonlinear problem as well. In order to state this result formally, we
need the following notation.

\begin{definition}\label{bidiffeo}
Let $U\subset\R^d$ be open. A $C^1$-map $W:U\times\R^d\to\R^d$ is said to be a uniform
local bidiffeomorphism, if there exist $r_0>0$, $y_0\in\R^d$ and a constant
$C>0$ such that for all $x\in U$ and $y\in B(y_0,r_0)$, the maps
$W(x,\cdot):B(y_0,r_0)\to W(x,B(y_0,r_0))$ and $W(\cdot,y):U\to W(U,y)$ are
uniform diffeomorphisms, that is, diffeomorphisms satisfying
\begin{equation}\label{bidiffeobound}
\Vert D_iW(x,y)\Vert\,,\,\Vert (D_iW(x,y))^{-1}\Vert\le C
\end{equation}
for all $x\in U$, $y\in B(y_0,r_0)$ and $i=1,2$, where $D_1W$ and $D_2W$ denote
the derivatives of $W(\cdot,y)$ and $W(x,\cdot)$, respectively. A uniform local
bidiffeomorphism $W$ generates a parameterised family of uniform diffeomorphisms
$W_x\colon B(y_0,r_0)\to W_x(B(y_0,r_0))$, $x\in U$, by the formula
$W_x(y):=W(x,y)$.
\end{definition}

\begin{lemma}\label{inverse}
Let $W_x\colon B(y_0,r_0)\to W_x(B(y_0,r_0))$, $x\in U$, be a parameterised
family of uniform diffeomorphisms generated by a uniform local
bidiffeomorphism $W:U\times\R^d\to\R^d$. Then there exists a parameterised
family of uniform diffeomorphisms $X_z\colon V_z\to X_z(V_z)$ where
${z}\in W(U,B(y_0,r_0))$ and $V_{{z}}\subset B(y_0,r_0)$ is open such that for all
$E\subset B(y_0,r_0)$, we have
\[
z\in W_x(E)\text{ if and only if }x\in X_{{z}}(E\cap V_{z}).
\]
Furthermore,
\begin{equation}\label{singfinal}
\Vert D X_z(y)\Vert\,,\,\Vert (D X_z(y))^{-1}\Vert\le C^2
\end{equation}
for all $z\in W(U,B(y_0,r_0))$ and $y\in V_z$. Here $C$ is as in
Definition~\ref{bidiffeo}.
\end{lemma}

\begin{proof}
Since for all $z\in W(U,B(y_0,r_0))$ the set
$U^z:=\{x\in U:z\in W(x,B(y_0,r_0))\}$
is open and non-empty, we may define a map $R^z:U^z\to B(y_0,r_0)$ by
$R^z(x):=T_x(z)$ where $T_x:=W(x,\cdot)^{-1}$. That is,
$$
W(x, R^z(x))=W(x, T_x(z))=z.
$$
 Consider $z\in W(U,B(y_0,r_0))$.
We show that $R^z:U^z\to R^z(U^z)$ is a uniform diffeomorphism.
If $x,u\in U^z$ satisfy $R^z(x)=R^z(u)$, then $T_x(z)=T_u(z)=y$ for some
$y\in B(y_0,r_0)$, and therefore, $W(x,y)=z=W(u,y)$. Thus $x=u$, implying
that $R^z$ is
injective. Since $W(x,T_x(z))=z$ for all $x\in U^z$, we have
$D_1W(x,T_x(z))+D_2W(x,T_x(z))\circ D_xT_x(z)=0$,
giving
\[
DR^z(x)=D_xT_x(z)=-\bigl(D_2W(x,T_x(z))\bigr)^{-1}\circ D_1W(x,T_x(z)).
\]
This implies
\begin{equation}\label{singularvalue}
\Vert DR^z(x)\Vert\,,\,\Vert (DR^z(x))^{-1}\Vert\le C^2
\end{equation}
for all $z\in W(U,B(y_0,r_0))$ and $x\in U^z$. Observing
that for all $E\subset B(y_0,r_0)$ and $x\in U$,
\[
z\in W(x,E)\iff T_x(z)\in E\iff R^z(x)\in E\iff x\in (R^{ z})^{-1}(E),
\]
we may define $V_{ z}:=R^{ z}(U^{ z})$ and $X_{ z}:=(R^{ z})^{-1}$. The claim \eqref{singfinal}
follows from \eqref{singularvalue}.
\end{proof}

For every $E\subset\R^d$ and $\delta>0$, let
\begin{equation}\label{dneighbourhood}
\overline V_\delta(E):=\{x\in\R^d:\dist(x,E)\le\delta\}
\end{equation}
be the closed $\delta$-neighbourhood of $E$. Here
$\dist(x,E):=\inf\{|x-a|:a\in E\}$ is the distance between $x$ and $E$.
According to the next lemma, using the notation of Lemma \ref{inverse},
for each Lebesgue measurable set $F\subset\R^d$, the
Lebesgue measure of $F\cap W_x(E)$ is close to that of $W_x(E)$ for most
points $x\in F$ provided that $E$ is a subset of a sufficiently small ball.

\begin{lemma}\label{prod}
Let $U\subset\R^d$, $r_0>0$, $y_0\in\R^d$ and $W:U\times\R^d\to\R^d$ be as in
Definition~\ref{bidiffeo}. Assume that $W_x(y_0)=x$ for all $x\in U$ and
$F\subset U$ is Lebesgue measurable.
Then for every $\varepsilon>0$, there is $\delta=\delta(F,\varepsilon)>0$ such
that for all Lebesgue measurable sets $E\subset B(y_0,\delta)$, we have
\begin{equation}\label{prodclaim}
\L\bigl(\bigl\{x\in F:\L(F\cap W_x(E))\ge (1-\varepsilon)\L(W_x(E))\bigr\}\bigr)
      \ge (1-\varepsilon)\L(F).
\end{equation}
\end{lemma}

\begin{proof}
We start by proving that $x\mapsto\L(F\cap W_x(E))$ is a Borel map. Assume
first that $F$ and $E$ are compact. Since $\L$ is a Radon
measure, we have $\L(A)=\lim_{\delta\to 0}\L(\overline V_\delta(A))$ for all
compact sets $A$. This, in turn, implies that the function $A\mapsto\L(A)$,
defined for compact sets, is upper semi-continuous. Moreover, the fact that
the map $A\mapsto A\cap E$ is upper semi-continuous for compact sets
$E\subset\R^d$ (for the definition of upper semi-continuity in this context see
\cite[p. 200]{Kechris}) implies that the map $x\mapsto\L(F\cap W_x(E))$ is
upper semi-continuous and, therefore, a Borel map.

Assume now that $F$ and $E$ are Lebesgue measurable. Since
$\mathcal L$ is inner regular, that is,
$\mathcal L(A)=\sup\{\mathcal L(C):C\subset A,\,C\text{ is compact}\}$
for all Lebesgue measurable sets $A\subset\mathbb R^d$, we may
choose increasing sequences $(F_i)_{i\in\mathbb N}$ and $(E_j)_{j\in\mathbb N}$
of compact sets such that $F_i\subset F$,
$E_j\subset E$, $\lim_{i\to\infty}\L(F_i)=\L(F)$ and
$\lim_{j\to\infty}\L(W_x(E_j))=\L(W_x(E))$ for all $x\in U$. In particular,
\[
\lim_{j\to\infty}\lim_{i\to\infty}\L(F_i\cap W_x(E_j))=\L(F\cap W_x(E))
\]
for all $x\in U$ and, therefore, the map $x\mapsto\L(F\cap W_x(E))$ is Borel
measurable. It follows that all the sets we encounter in the proof below
are Lebesgue measurable.

First we prove \eqref{prodclaim} for compact sets $F$. Clearly,
we may assume that $\L(F)>0$. Note that \eqref{prodclaim} is equivalent to
\begin{equation}\label{smallset}
\L\bigl(\bigl\{x\in F:\L(F^c\cap W_x(E))>\varepsilon\L(W_x(E))\bigr\}\bigr)
      <\varepsilon\L(F),
\end{equation}
where the complement of a set $A$ is denoted by $A^c$.
Now  suppose that  \eqref{prodclaim} is not true. Then there exists $\varepsilon>0$ such
that for all $\delta>0$, there is a measurable set $E\subset B(y_0,\delta)$
with $\L(E)>0$ satisfying $\L(\Lambda)\ge\varepsilon\L(F)$, where
\[
\Lambda:=\{x\in F:\L(F^c\cap W_x(E))>\varepsilon\L(W_x(E))\}.
\]
Suppose that $z\in W_x(E)$. Since $W_x(y_0)=x$ for all $x\in U$, we have
$|z-x|\le C_2\delta=:\tilde\delta$. Denoting the
characteristic function of a set $A$ by $\chi_A$, we obtain by Fubini's theorem
that
\begin{equation}\label{complement}
\begin{split}
\int_\Lambda\L(F^c\cap W_x(E))\,d\L(x)&\le\int_F\L(F^c\cap W_x(E))\,d\L(x)\\
&=\iint\chi_F(x)\chi_{W_x(E)}(z)\chi_{F^c}(z)\,d\L(z)d\L(x)\\
&=\iint\chi_F(x)\chi_{W_x(E)}(z)\chi_{\overline V_{\tilde\delta}(F)\setminus F}(z)\,
  d\L(z)d\L(x)\\
&=\int_{\overline V_{\tilde\delta}(F)\setminus F}\int_F\chi_{W_x(E)}(z)\,d\L(x)d\L(z).
\end{split}
\end{equation}
From Lemma~\ref{inverse} we deduce that $z\in W_x(E)$ if and only if
$x\in X_{z}(E\cap V_{z})$. Furthermore, $\L(X_{z}(E\cap V_{z}))\le C^{2d}\L(E)$ by
\eqref{singfinal}. Thus
\begin{equation}\label{upper}
\int_\Lambda\L(F^c\cap W_x(E))\,d\L(x)
  \le C^{2d}\L(E)\L(\overline V_{\tilde\delta}(F)\setminus F).
\end{equation}

On the other hand, since $W_x$ is a uniform diffeomorphism,
$\L(W_x(E))\ge C^{-d}\L(E)$ for all $x\in U$. Combining this with the
definition of $\Lambda$, inequality \eqref{upper} and
the fact that $\L(\Lambda)\ge\varepsilon\L(F)$, we obtain
\begin{equation}\label{lower}
\begin{split}
&\int_\Lambda\L(F^c\cap W_x(E))\,d\L(x)
  \ge\varepsilon\int_\Lambda\L(W_x(E))\,d\L(x)\\
&=\varepsilon\int_\Lambda\int\chi_{W_x(E)}({z})\,d\L({z})d\L(x)
  \ge\varepsilon\int_\Lambda\int\chi_{W_x(E)}({z})\chi_F({z})\,d\L({z})d\L(x)\\
&=\varepsilon\int_\Lambda\int\chi_{W_x(E)}({z})\,d\L({z})d\L(x)-
  \varepsilon\int_\Lambda\int\chi_{W_x(E)}({z})\chi_{F^c}({z})\,d\L({z})d\L(x)\\
&\ge C^{-d}\varepsilon\L(E)\L(\Lambda)-
  \varepsilon\int_\Lambda\int\chi_{W_x(E)}({z})\chi_{F^c}({z})\,d\L({z})d\L(x)\\
&\ge\varepsilon\L(E)\bigl(C^{-d}\varepsilon\L(F)-
  C^{2d}\L(\overline V_{\tilde\delta}(F)\setminus F)\bigr).
\end{split}
\end{equation}
Since $F$ is compact, $\L(F)=\lim_{i\to\infty}\L(\overline V_{\frac 1i}(F))$ and,
therefore, for every $\tilde\varepsilon>0$, there is $\delta>0$
such that
$\L(\overline V_{\tilde\delta}(F)\setminus F)<\tilde\varepsilon\L(F)$.
Hence, \eqref{lower} contradicts  \eqref{upper},
completing the proof of \eqref{prodclaim} for compact sets $F$.

For a Lebesgue measurable set $F$, choose a compact set $K\subset F$
satisfying
$\L(K)\ge (1-\varepsilon)\L(F)$. Then
\begin{align*}
&\L\bigl(\bigl\{x\in F:\L(F\cap W_x(E))\ge (1-\varepsilon)^2\L(W_x(E))\bigr\}
   \bigr)\\
&\ge\L\bigl(\bigl\{x\in K:\L(K\cap W_x(E))\ge (1-\varepsilon)\L(W_x(E))\bigr\}
   \bigr)\\
&\ge (1-\varepsilon)\L(K)\ge(1-\varepsilon)^2\L(F),
\end{align*}
completing the proof of \eqref{prodclaim}.
\end{proof}

The last lemma of this section is a counterpart of Lemma~\ref{prod} for
energies of sets.

\begin{lemma}\label{intestimate}
Let $U\subset\R^d$, $r_0>0$, $y_0\in\R^d$ and $W:U\times\R^d\to\R^d$ be as in
Definition~\ref{bidiffeo}. Assume that $W_x(y_0)=x$ for all $x\in U$. Let
$A_1,A_2\subset U$ be bounded Lebesgue measurable sets and let
$0\le t<d$. Then for every $\e>0$, there exists
$\delta_1=\delta_1(A_1,A_2,\varepsilon)>0$ such that
\[
\begin{split}
\iint_{A_1\times A_2}&\iint_{W_{x_1}(E_1)\times W_{x_2}(E_2)}|u_1-u_2|^{-t}\,
  d\L(u_1)d\L(u_2)d\L(x_1)d\L(x_2)\\
&\le(1+\e)\iint_{A_1\times A_2}\L(W_{x_1}(E_1))\L(W_{x_2}(E_2))|x_1-x_2|^{-t}\,d\L(x_1)
  d\L(x_2),
\end{split}
\]
provided that $E_1,E_2\subset B(y_0,\delta_1)$ are Lebesgue measurable.
\end{lemma}

\begin{proof}
Clearly, we may assume that $\L(A_1)>0$ and $\L(A_2)>0$. Let $R>1$ be such that
$A_1, A_2\subset B(0,R-1)$. Then
\[
0<\iint_{A_1\times A_2}\vert u_1-u_2\vert^{-t}\,d\L(u_1)d\L(u_2)
\le\iint_{B(0,R)\times B(0,R)}\vert u_1-u_2\vert^{-t}\,d\L(u_1)d\L(u_2)<\infty.
\]
It follows
that for every $\tilde\e>0$, there exists $\delta\in\R$ with $0<\delta<1$ such
that
\begin{equation}\label{defdelta}
\iint_{D(\delta)}|u_1-u_2|^{-t}\,d\L(u_1)d\L(u_2)
\le\tilde\e\iint_{A_1\times A_2}\vert u_1-u_2\vert^{-t}\,d\L(u_1)d\L(u_2),
\end{equation}
where $D(\delta):=\{(u_1,u_2)\in B(0,R)\times B(0,R):|u_1-u_2|\le\delta\}$.
Consider $0<\tilde\varepsilon<\frac 12$ and let $\delta>0$ be such that
\eqref{defdelta} is valid. Defining
$\delta_1:=\frac 14C^{-1}\tilde\e\delta$, gives
$\diam W_x(B(y_0,\delta_1))<\frac 12\tilde\e\delta$ for all $x\in U$.

Let $E_1$ and $E_2$ be Lebesgue measurable subsets of $B(y_0,\delta_1)$. Recall
that $W_x(y_0)=x$ for all $x\in U$. Thus, if
$u_i\in W_{x_i}(E_i)$ for $i=1,2$ and $|x_1-x_2|>\frac 12\delta$, we have
$|u_1-u_2|>(1-2\tilde\e)|x_1-x_2|$ and, therefore,
\begin{equation}\label{bigdistance}
\begin{split}
&\int_{\{(x_1,x_2)\in A_1\times A_2\;:\;|x_1-x_2|>\frac 12\delta\}}
  \iint_{W_{x_1}(E_1)\times W_{x_2}(E_2)}|u_1-u_2|^{-t}\,d\L(u_1)d\L(u_2)d\L(x_1)d\L(x_2)\\
&\phantom{u}\le (1-2\tilde\e)^{-t}\int_{\{(x_1,x_2)\in A_1\times A_2\,:\,
  |x_1-x_2|>\frac 12\delta\}}|x_1-x_2|^{-t}\\
&\phantom{eeeeeeeeeeeeeeeee}\times\L(W_{x_1}(E_1))
  \L(W_{x_2}(E_2))\,d\L(x_1)d\L(x_2)\\
&\phantom{u}\le (1-2\tilde\e)^{-t}\iint_{A_1\times A_2}\L(W_{x_1}(E_1))\L(W_{x_2}(E_2))
  |x_1-x_2|^{-t}\,d\L(x_1)d\L(x_2).
\end{split}
\end{equation}

To estimate the remaining part of the integral, we make
the change of variables $u_i=W_{x_i}(\tilde u_i)=W(x_i,\tilde u_i)$ for $i=1,2$.
The Jacobians of these coordinate transformations are bounded from above by
$C^d$. By Fubini's theorem,
\[
\begin{split}
&\int_{\{(x_1,x_2)\in A_1\times A_2\,:\,|x_1-x_2|\le\frac 12\delta\}}
   \iint_{W_{x_1}(E_1)\times W_{x_2}(E_2)}|u_1-u_2|^{-t}\,d\L(u_1)d\L(u_2)
   d\L(x_1)d\L(x_2)\\
&\phantom{a}\le C^{2d}\iint_{E_1\times E_2}
   \int_{\{(x_1,x_2)\in A_1\times A_2\,:\,|x_1-x_2|\le\frac 12\delta\}}
   |W(x_1,\tilde u_1)-W(x_2,\tilde u_2)|^{-t}\\
&\phantom{aaaaaaaaaaaaaaaaaaaaaa}\times d\L(x_1)d\L(x_2)d\L(\tilde u_1)
   d\L(\tilde u_2)=:L.
\end{split}
\]
Recall that by the choice of $\delta_1$, we have
$|W(x_1,\tilde u_1)-W(x_2,\tilde u_2)|\le\delta$ provided that
$|x_1-x_2|\le\frac 12\delta$. The fact that for $i=1,2$ we have
$|W(x_i,\tilde u_i)-W(x_i,y_0)|\le C\delta_1<1$ for all $x_i\in A_i$ and
$\tilde u_i\in E_i$ gives $W(x_i,\tilde u_i)\in B(0,R)$. Making the
change of variables $\tilde x_i=W(x_i,\tilde u_i)$ for $i=1,2$ and using the
fact that the Jacobians are bounded by $C^d$, inequality \eqref{defdelta} gives
\begin{equation}\label{smalldistance}
\begin{split}
L &\le C^{4d}\iint_{E_1\times E_2}\iint_{D(\delta)}
  |\tilde x_1-\tilde x_2|^{-t}\,d\L(\tilde x_1)d\L(\tilde x_2)d\L(\tilde u_1)
  d\L(\tilde u_2)\\
&\le C^{4d}\tilde\e\L(E_1)\L(E_2)\iint_{A_1\times A_2}
  |\tilde x_1-\tilde x_2|^{-t}\,d\L(\tilde x_1)d\L(\tilde x_2)\\
&\le C^{6d}\tilde\e\iint_{A_1\times A_2}\L(W_{x_1}(E_1))\L(W_{x_2}(E_2))
  |\tilde x_1-\tilde x_2|^{-t}\,d\L(\tilde x_1)d\L(\tilde x_2).
\end{split}
\end{equation}
Combining \eqref{bigdistance} and \eqref{smalldistance}, gives the claim.
\end{proof}

\section{Upper bound for Hausdorff dimension}\label{upperbound}

In this section, we prove   Theorem ~\ref{main}.(a) and a key equality in  Theorem ~\ref{main}.(c).  We begin with the following observation.

\begin{lemma}\label{H_inftyupperbound}
Let $(E_n)_{n\in\N}$ be a sequence of subsets of $\R^d$. Then
\[
\dimH\bigl(\limsup_{n\to\infty}E_n\bigr)\le\inf\bigl\{t\ge 0:\sum_{n=1}^\infty
  \mathcal H_\infty^t(E_n)<\infty\bigr\}.
\]
\end{lemma}

\begin{proof}
For $0\le s<\infty$ and $0<\delta<\infty$, we denote by $\mathcal H^s$ and
$\mathcal H_\delta^s$ the  $s$-dimensional Hausdorff measure and
$\delta$-measure, respectively.
Let $t>0$ with $\sum_{n=1}^\infty\mathcal H_\infty^t(E_n)<\infty$. For the
purpose of proving the
claim, it suffices to show that $\dimH\bigl(\limsup_{n\to\infty}E_n\bigr)\le t$.
In what follows, we prove a slightly stronger result that
$\mathcal H^t\bigl(\limsup_{n\to\infty}E_n\bigr)=0$.

Let $\varepsilon>0$ and let $N\in\mathbb N$ so that
$\sum_{n=N}^\infty\mathcal H_\infty^t(E_n)<\frac{\varepsilon^t}2$.
For every $n\geq N$ and $k\in\N$, we choose $U_{n,k}\subset\R^d$ such that
$\bigcup_{k=1}^\infty U_{n,k}\supset E_n$ for all $n\geq N$ and
$\sum_{n=N}^\infty\sum_{k=1}^\infty(\diam U_{n,k})^t\le\varepsilon^t$. Clearly,
$\diam U_{n,k}\leq\varepsilon$, and therefore,
\[
\mathcal H_{\varepsilon}^t\bigl(\limsup_{n\to\infty}E_n\bigr)
\le\mathcal H_{\varepsilon}^t\bigl(\bigcup_{n=N}^\infty E_n\bigr)
\le\sum_{n=N}^\infty\sum_{k=1}^\infty(\diam U_{n,k})^t\le\varepsilon^t.
\]
As $\varepsilon$ can be arbitrarily small, we have
$\mathcal H^t(\limsup_{n\to\infty}E_n)=0$, which completes the proof.
\end{proof}

\begin{proof}[Proof of Theorem ~\ref{main}.(a)]
The inequality  $\dimH\bE(\x,\,\bA)\le t_0(\bA)$  follows directly from Lemma  \ref{H_inftyupperbound}, using a simple
observation that $\mathcal H^t_\infty(f(x,E))\leq (C_u)^t\mathcal H^t_\infty(E)$ for all $x\in  U$  and $E \subset V$, where $C_u$
is the constant appearing in \eqref{diffeobound}.
\end{proof}

The rest of this section is devoted to proving that $s_0(\bA)=t_0(\bA)$ under
the assumptions of Theorem~\ref{main}.(c), where $s_0(\bA)$ and $t_0(\bA)$ are
as in \eqref{defs0} and \eqref{deft0}, respectively. We start by proving that
$s_0(\bA)\le t_0(\bA)$.

\begin{lemma}\label{HgeG}
Let $E\subset\R^d$ be a Lebesgue measurable set. For all
$t\ge 0$, we have $\mathcal H_\infty^t(E)\ge G_t(E)$. In particular,
for every sequence $\bA:=(A_n)_{n\in\mathbb N}$ of Lebesgue measurable subsets of
$\R^d$, we have $s_0(\bA)\le t_0(\bA)$.
\end{lemma}

\begin{proof}
We may assume that $\L(E)>0$ and $\mathcal H^t_\infty(E)<\infty$.
Let $\varepsilon>0$. For every $n\in\N$, we choose disjoint Borel
sets $E_n$ such that $\bigcup_{n=1}^\infty E_n\supset E$ and
$\sum_{n=1}^\infty(\diam E_n)^t<\mathcal H^t_\infty(E)+\varepsilon$. Notice that
for all $n\in\mathbb N$,
\[
I_t(E\cap E_n)=\iint_{(E\cap E_n)\times (E\cap E_n)}|x-y|^{-t}\, d\L(x)d\L(y)\ge
 (\diam E_n)^{-t}\L(A\cap E_n)^2.
\]
It follows that
\begin{equation}\label{Itestimate}
I_t(E)\ge\sum_{n=1}^\infty I_t(E\cap E_n)\ge\sum_{n=1}^\infty(\diam E_n)^{-t}
  \L(E\cap E_n)^2.
\end{equation}
From \eqref{Itestimate} and Cauchy-Schwarz inequality, we obtain
\begin{align*}
\bigl(\sum_{n=1}^\infty(\diam E_n)^t\bigr)I_t(E) &\ge\bigl(\sum_{n=1}^\infty
 (\diam E_n)^t\bigr)\bigl(\sum_{n=1}^\infty(\diam E_n)^{-t}\L(E\cap E_n)^2\bigr)\\
 &\ge\bigl(\sum_{n=1}^\infty\L(E\cap E_n)\bigr)^2=\L(E)^2,
\end{align*}
which implies that $\sum_{n=1}^\infty(\diam E_n)^t\ge g_t(E)$ (see \eqref{gtdef}).
Hence,
$\mathcal H_\infty^t(E)+\varepsilon>g_t(E)$. Letting $\varepsilon$ tend to zero,
gives $\mathcal H_\infty^t(E)\ge g_t(E)$. As $\mathcal H_\infty^t(\cdot)$
is a monotone increasing function, we conclude that
$\mathcal H_\infty^t(E)\geq G_t(E)$. According to this inequality, we have
$s_0(\bA)\leq t_0(\bA)$ for every sequence $\bA$
of Lebesgue measurable subsets of $\R^d$.
\end{proof}

\begin{remark}\label{contentbound}
The  following extension of Lemma~\ref{HgeG} can be proven with the same
argument: if $\mu$ is a finite Borel measure supported on
$E$ and $t\geq 0$, we have
\[
\mathcal H_\infty^t(E)\ge\frac{\mu(E)^2}{\iint_{E\times E}|x-y|^{-t}\,
  d\mu(x)d\mu(y)}.
\]
\end{remark}

We proceed by estimating $\mathcal H_\infty^t(E)$ from above by means of
$G_t(E)$. Our estimate is based on a technical result stated in
Proposition~\ref{thm:energy_lemma}. In what follows, the restriction of
a measure
$\mu$ to a set $E\subset\mathbb R^d$ is denoted by $\mu\vert_E$, that is,
$\mu\vert_E(F):=\mu(E\cap F)$ for all $F\subset\mathbb R^d$.
For a Radon measure $\mu$ on $\mathbb R^d$ and $0<s<d$, let
\[
I_s(\mu):=\iint\vert x-y\vert^{-s}\,d\mu(x)d\mu(y)
\]
be the $s$-energy of $\mu$. Given a Borel set $E\subset\R^d$, let $\CP(E)$
be the space of Borel probability measures supported on $E$, and let $E^+$ be
the set of points in $E$ having positive Lebesgue density, that is,
\[
E^+:=\bigl\{x\in E:\; \liminf_{r\to 0}\frac{\L(E\cap B(x, r))}{\L(B(x,r))}>0
\bigr\}.
\]
We denote by $\overline E$ the closure of a set $E\subset\R^d$, by
$\overline B(0,1)\subset\R^d$ the closed unit ball centred at the
origin and by $\mathcal C(\overline B(0,1))$ the family of continuous maps from
$\overline B(0,1)$ to $\R$.

We continue by verifying several elementary lemmas.

\begin{lemma}\label{lemma-1}
Letting $s>0$, the mapping $\eta\mapsto I_s(\eta)$ is lower semi-continuous
on $\CP(\overline B(0,1))$.
\end{lemma}

\begin{proof}
The result is well known (see for example \cite[(1.4.5)]{Lan72}) and follows
from the fact that the mapping $(x,y)\mapsto |x-y|^{-s}$ is non-negative and
lower semi-continuous on $\R^d\times\R^d$.
\end{proof}

\begin{lemma}\label{lemma-3}
Let $\eta\in \CP(\overline B(0,1))$. Suppose that $(F_n)_{n\in\N}$ is a sequence
of Borel subsets of $\overline B(0,1)$ satisfying $\lim_{n\to\infty}\eta(F_n)=1$.
Then $\eta_n:=\eta(F_n)^{-1}\eta|_{F_n}$ converges to $\eta$ in the weak-star
topology as $n$ tends to infinity.
Moreover, $\lim_{n\to \infty} I_s(\eta_n)=I_s(\eta)$ for all $s>0$.
\end{lemma}

\begin{proof}
Letting $g\in\mathcal C(\overline B(0,1))$, it follows from Lebesgue's
dominated convergence theorem that
\[
\lim_{n\to\infty}\int g\;d\eta_n=\lim_{n\to\infty}\eta(F_n)^{-1}
  \int g\chi_{F_n}\;d\eta=\int g\;d\eta,
\]
and therefore, $\eta_n$ converges to $\eta$ in the weak-star topology.

Let $s>0$. By Lemma \ref{lemma-1}, we have
$\liminf_{n\to\infty}I_s(\eta_n)\geq I_s(\eta)$. Notice that for all $n\in\N$,
\[
I_s(\eta_n)=\eta(F_n)^{-2}\iint_{F_n\times F_n} |x-y|^{-s}\;d\eta(x)
d\eta(y)\leq\eta(F_n)^{-2} I_s(\eta),
\]
which implies that $\limsup_{n\to\infty}I_s(\eta_n)\leq I_s(\eta)$. Hence,
$\lim_{n\to\infty}I_s(\eta_n)=I_s(\eta)$, as desired.
\end{proof}

For a Borel set $F\subset\overline B(0,1)$ and $s>0$, we recall the notation
$I_s(F)=I_s(\L|_F)$ from \eqref{setenergy}. For every $k\in\N$, define
\begin{equation}
\label{e-CQ}
\CQ_k:=\{[0,2^{-k})^d+\alpha:\alpha\in 2^{-k}\Z^d\}.
\end{equation}

\begin{lemma}\label{lem-Leb}
Let $F\subset\overline B(0,1)$ be a Borel set, and let $0<s<d$. Then for every
$p\in\R$ with $0<p\leq 1$, there exists a Borel set $F_1\subset F$ so that
$\L(F_1)=p\L(F)$ and $I_s(F_1)\leq 2p^2I_s(F)$.
\end{lemma}

\begin{proof}
Let $0<p\leq 1$. Write $\mu:=\L\vert_F$ and choose a large integer $\ell\in\N$
so that
\begin{equation}\label{e-u0}
(1+\tfrac{2\sqrt{d}}{\ell})^s<\tfrac32.
\end{equation}
Since $I_s(\mu)<\infty$, there is $n\in\N$ such that
\begin{equation}\label{e-u1}
 \sum_{\substack{Q,Q'\in\CQ_n \\ \dist(Q,Q')< 2^{-n}\ell}}
  \iint_{Q\times Q'}|x-y|^{-s}\,d\mu(x)d\mu(y)<\tfrac12 p^2I_s(\mu).
\end{equation}
Here $\dist(Q,Q')=\inf\{\vert x-y\vert:x\in Q\text{ and }y\in Q'\}$.
For each $Q\in\CQ_n$, construct a Borel subset $\widetilde{Q}$ of
$Q\cap F$ such that $\L(\widetilde{Q})=p\L({Q\cap F})$.  Defining
$F_1:=\bigcup_{Q\in  \CQ_n} \widetilde{Q}$, we have $F_1\subset F$
and $\L(F_1)=p\L(F)$.

We proceed by showing that $I_s(F_1)\leq 2p^2 I_s(F)$.
Set $\eta:=\L\vert_{F_1}$. Since $F_1\subset F$, inequality \eqref{e-u1} gives
\begin{equation*}\label{e-u2}
 \sum_{\substack{Q,Q'\in\CQ_n \\\dist(Q,Q')< 2^{-n}\ell}}
  \iint_{Q\times Q'}|x-y|^{-s}\,d\eta(x)d\eta(y)<\tfrac12p^2I_s(\mu).
\end{equation*}
The proof will be complete, once we show that
\begin{equation}\label{e-u3}
 \sum_{\substack{Q,Q'\in\CQ_n \\ \dist(Q,Q')\geq 2^{-n}\ell}}
  \iint_{Q\times Q'}|x-y|^{-s}\,d\eta(x)d\eta(y)\leq\tfrac32 p^2I_s(\mu).
\end{equation}
Note that if $Q,Q'\in\CQ_n$ with $\dist(Q,Q')\geq 2^{-n}\ell$,
$x\in Q$ and $y\in Q'$, we obtain
\[
\dist(Q,Q')\leq |x-y|\leq\dist(Q,Q')+2\sqrt{d}2^{-n}
\]
and, therefore, by \eqref{e-u0},
\[
\tfrac23\dist(Q,Q')^{-s}\leq |x-y|^{-s}\leq\dist(Q,Q')^{-s}.
\]
This, in turn, implies that
\begin{align*}
\iint_{Q\times Q'}&|x-y|^{-s}\,d\eta(x)d\eta(y)\leq\dist(Q,Q')^{-s}\eta(Q)
\eta(Q')\\
&=p^2\dist(Q,Q')^{-s}\mu(Q)\mu(Q')
\leq\tfrac32p^2 \iint_{Q\times Q'}|x-y|^{-s}\,d\mu(x)d\mu(y).
\end{align*}
Summing over $Q,Q'\in\CQ_n$ with $d(Q,Q')\geq 2^{-n}\ell$, we obtain
\eqref{e-u3} as desired.
\end{proof}

The following lemma is a special case of Proposition~\ref{thm:energy_lemma}.

\begin{lemma}\label{lemma-5}
Let $E\subset\overline B(0,1)$ be a Borel set with $\L(E)>0$, and let
$k,m\in\N$. Assume that $E_0\subset E$ is a non-empty Borel set such that
\begin{equation}\label{lemma-5-assump}
\frac{\L(E\cap B(x,r))}{\L(B(x,r))}>\frac{1}{k}
\end{equation}
for all $x\in E_0$ and $0<r\leq 2^{-m}$.
Let $0<s<d$ and $\mu\in\CP(E_0)$ with $I_s(\mu)<\infty$.
Then there is a sequence $(F_n)_{n\in\N}$ of Borel subsets of $E$
with positive  Lebesgue measure such that
the sequence $\mu_n:=\mathcal L(F_n)^{-1}\mathcal L\vert_{F_n}$, $n\in\N$,
converges to $\mu$ in the weak-star topology as $n$ tends to infinity, and
$\lim_{n\to\infty}I_s(\mu_n)=I_s(\mu)$.
\end{lemma}

\begin{proof} We divide the proof into three steps.

{\sl Step 1. Construction of $(\mu_n)_{n\in\mathbb N}$}. For all $n\in\N$, let
\[
\{x_{n,1},\dots,x_{n,p_n} : |x_{n,i}-x_{n,j}|\geq 2^{-n}\text{ for all }i\neq j\}
\]
be a subset of $E_0$ with maximal cardinality. Then
\[
E_0\subset\bigcup_{i=1}^{p_n} B(x_{n, i},2^{-n}).
\]
For $i=1,\ldots, p_n$, we denote by $Q_{n,i}$ the set of points
$y\in B(x_{n,i},2^{-n})$ for which $i$ is the smallest index such that
$|y-x_{n,i}|=\min_{j=1,\dots,p_n}|y-x_{n,j}|$. Then the sets
$Q_{n,i}$, $i=1,\ldots,p_n$, are pairwise disjoint Borel sets satisfying
\begin{align}
\label{cube1} E_0&\subset\bigcup_{i=1}^{p_n} Q_{n,i}=\bigcup_{i=1}^{p_n} B(x_{n, i},
2^{-n})\text{ and}\\
\label{cube2}B(x_{n,i},2^{-n-1})&\subset Q_{n,i}\subset B(x_{n,i},2^{-n})
             \text{ for all }i=1,\dots,p_n.
\end{align}
For all $i=1,\ldots,p_n$, define
\begin{equation}\label{cube0}
c_i:=\frac{\mu(Q_{n,i})}{\L(E\cap B(x_{n,i},2^{-n-2}))}
\end{equation}
and set $c:=\max_{i=1,\ldots, p_n}c_i$.
Lemma \ref{lem-Leb} implies that for every $i=1,\ldots, p_n$, we can construct
a Borel set $F_{n,i}$ such that
\begin{align}
\label{cube3'}F_{n,i}& \subset E\cap B(x_{n,i},2^{-n-2}),\\
\label{cube3} \L(F_{n,i})& =\tfrac{c_i}{c}\L\left(E\cap B(x_{n,i},2^{-n-2})\right)
\text{ and }\\
\label{cube4} I_s(F_{n,i})& \leq \tfrac{2 c_i^2}{c^2}
I_s\left(E\cap B(x_{n,i},2^{-n-2})\right).
\end{align}
By \eqref{cube3'}, the sets $F_{n,i}$, $i=1,\ldots,p_n$, are pairwise disjoint
and, moreover,
\begin{equation}\label{e-cube0}
\dist(F_{n,i},F_{n,j})\geq 2^{-n-1}\text{ for }i\neq j.
\end{equation}
We complete the construction in step 1 by setting
\[
F_n:=\bigcup_{i=1}^{p_n} F_{n, i}\quad\text{ and }
\quad\mu_n:=\L(F_n)^{-1}\mathcal L\vert_{F_n}.
\]
Observe that $\L(F_n)>0$ since $\L(B(x,r)\cap E)>0$ for all $x\in E_0$ and
$r>0$.

{\sl Step 2. Convergence of $(\mu_n)_{n\in\mathbb N}$.}
By \eqref{cube0} and \eqref{cube3},
we have $\L(F_{n,i})=c^{-1}\mu(Q_{n,i})$ for all $i=1,\ldots,p_n$. It follows
that
\begin{align}
\label{Fnmeasure}\L(F_n)&=c^{-1}\text{ and}\\
\label{e-equal}\mu_n(Q_{n,i})&=\mu_n(F_{n,i})=\mu(Q_{n,i})
\end{align}
for all $i=1,\ldots, p_n$. Let $F\subset\R^d$ be a compact set. From
\eqref{e-equal} and the fact that $\diam(Q_{n,i})\leq 2\cdot 2^{-n}$ (see
\eqref{cube2}), we conclude that for all $\e>0$, (recall
\eqref{dneighbourhood})
\[
\limsup_{n\to\infty}\mu_n(F)\le\limsup_{n\to\infty}\sum_{\substack{
  1\le i\le p_n\\ Q_{n,i}\subset F_\e}}\mu_n(Q_{n,i})\le\mu(\overline V_\e(F)).
\]
Combining this with the fact that
$\mu(F)=\lim_{\varepsilon\to 0}\mu(\overline V_\e(F))$, leads to the conclusion
$\limsup_{n\rightarrow\infty}\mu_n(F)\le\mu(F)$. The  weak-star convergence now
follows from the Portmanteau theorem \cite[Theorem 17.20]{Kechris}.

{\sl Step 3. Convergence of $(I_s(\mu_n))_{n\in\mathbb N}$.}
Since the sequence $(\mu_n)_{n\in\N}$
converges to $\mu$ in the weak-star topology, Lemma \ref{lemma-1} gives
$\liminf_{n\to \infty}I_s(\mu_n)\geq I_s(\mu)$. Hence, for the purpose of proving
that $\lim_{n\to\infty}I_s(\mu_n)=I_s(\mu)$, it suffices to show that for every
$\varepsilon>0$, there exists $N\in\mathbb N$ such that
\begin{equation}\label{e-N}
I_s(\mu_n)\leq (1+\varepsilon)I_s(\mu)
\end{equation}
for all $n\geq N$.
Let $\varepsilon>0$ and select $\ell\in\N$ such that
\begin{equation}\label{e-u00}
(1+\tfrac4\ell)^s<1+\tfrac\varepsilon2.
\end{equation}
Moreover, choose a large integer $N\geq m$ such that for all $n\geq N$,
\begin{equation}\label{e-T1}
\iint_{\{(x,y)\;:\; |x-y|\leq 2^{-n}(\ell+8)\}} |x-y|^{-s}\;d\mu(x)d\mu(y)
<\tfrac{\varepsilon}{4 L} I_s(\mu)
\end{equation}
where
\begin{equation*}
L:=\max\left\{2^s (\ell+8)^s, \; 2k^2 I_s(B(0,1)) \L(B(0,1))^{-2} 8^s\right\}.
\end{equation*}
Let $n\geq N$ and set $\D_n:=\{Q_{n,i}: i=1,\ldots, p_n\}$. Notice that if
$Q,Q'\in{\bf D}_n$ with $\dist(Q,Q')\geq 2^{-n}\ell$, $x\in Q$ and $y\in Q'$,
we have by \eqref{cube2} that
\[
\dist(Q,Q')\leq |x-y|\leq\dist(Q,Q')+4\cdot 2^{-n}
\]
and, therefore, by \eqref{e-u00},
\[
(1+\tfrac\varepsilon2)^{-1}\dist(Q,Q')^{-s}\leq |x-y|^{-s}\leq \dist(Q,Q')^{-s}.
\]
Combining this with \eqref{e-equal}, we conclude that
\begin{align*}
\iint_{Q\times Q'}&|x-y|^{-s}\,d\mu_n(x)d\mu_n(y)\leq\dist(Q,Q')^{-s}
  \mu_n(Q)\mu_n(Q')\\
&=\dist(Q,Q')^{-s}\mu(Q)\mu(Q')
  \leq (1+\tfrac\varepsilon2)\iint_{Q\times Q'}|x-y|^{-s}\,d\mu(x)d\mu(y).
\end{align*}
Summing over $Q, Q'\in {\bf D}_n$ with $\dist(Q,Q')\geq 2^{-n}\ell$, we obtain
that
\begin{equation*}\label{e-u5}
 \sum_{\substack{Q,Q'\in{\bf D}_n \\ \dist(Q,Q')\geq 2^{-n}\ell}}
  \iint_{Q\times Q'}|x-y|^{-s}\,d\mu_n(x)d\mu_n(y)\leq(1+\tfrac\varepsilon2)
I_s(\mu).
\end{equation*}

To complete the proof of \eqref{e-N}, it is sufficient to verify that
\begin{equation}\label{e-u6}
\sum_{\substack{Q,Q'\in{\bf D}_n \\\dist(Q,Q')< 2^{-n}\ell}}
  \iint_{Q\times Q'}|x-y|^{-s}\,d\mu_n(x)d\mu_n(y)\leq\tfrac\varepsilon2I_s(\mu).
\end{equation}
Since $\mu_n$ is supported on $F_n=\bigcup_{i=1}^{p_n}F_{n,i}$, the left-hand
side of \eqref{e-u6} is bounded above by
\begin{equation*}\label{e-u7}
 \sum_{\substack{1\leq i,j\leq p_n \\\dist(F_{n,i}, F_{n,j})< 2^{-n}(\ell+4)}}
  \iint_{F_{n,i}\times F_{n,j}}|x-y|^{-s}\,d\mu_n(x)d\mu_n(y)=:(I)+(II)
\end{equation*}
where
\begin{align*}
 (I)&:=\sum_{\substack{1\leq i,j\leq p_n:\; i\neq j \\\dist(F_{n,i}, F_{n,j})< 2^{-n}(\ell+4)}}
  \iint_{F_{n,i}\times F_{n,j}}|x-y|^{-s}\,d\mu_n(x)d\mu_n(y)\,\text{ and } \\
 (II)&:=\sum_{1\leq i\leq p_n}
  \iint_{F_{n,i}\times F_{n,i}}|x-y|^{-s}\,d\mu_n(x)d\mu_n(y).
\end{align*}
We proceed by estimating (I) and (II) separately. First we obtain
\begin{equation*}
\begin{aligned}
(I)&\leq\sum_{\substack{1\leq i,j\leq p_n:\; i\neq j \\\dist(F_{n,i}, F_{n,j}) < 2^{-n}(\ell+4)}}
2^{(n+1)s} \mu_n(F_{n,i})\mu_n(F_{n,j})&(\text{by \eqref{e-cube0}})\\
&=\sum_{\substack{1\leq i,j\leq p_n:\; i\neq j \\\dist(F_{n,i}, F_{n,j}) < 2^{-n}(\ell+4)}}
2^{(n+1)s} \mu(Q_{n,i})\mu(Q_{n,j})&(\text{by \eqref{e-equal}})\\
&\leq\sum_{\substack{Q,Q'\in{\bf D}_n \\\dist(Q,Q')< 2^{-n}(\ell+4)}} 2^{(n+1)s}
\mu(Q)\mu(Q')\\
&\leq\sum_{\substack{Q,Q'\in{\bf D}_n \\\dist(Q,Q')< 2^{-n}(\ell+4)}} 2^s (\ell+8)^s
\iint_{Q\times Q'} |x-y|^{-s} \;d\mu(x)d\mu(y)&(\text{by \eqref{cube2}})\\
&\leq  2^s (\ell+8)^s \iint_{\{ (x,y)\;:\; |x-y|\leq 2^{-n}(\ell+8)\}}
|x-y|^{-s} \;d\mu(x)d\mu(y)\\
& \leq \tfrac\varepsilon4 I_s(\mu).&(\text{by \eqref{e-T1}})
\end{aligned}
\end{equation*}
Set $\alpha:=\L(B(0,1))$ and $\beta:=I_s(B(0,1))$. Using the change of
variables $\tilde x=rx$, it is straightforward to see that
\begin{equation}\label{changevar}
I_s(B(x,r))=\mathcal L(B(x,r))^2\alpha^{-2}r^{-s}\beta
\end{equation}
for all $x\in\mathbb R^d$ and all $r>0$. Therefore,
\begin{equation*}
\begin{aligned}
(II) &=\sum_{1\le i\le p_n} c^2I_s(F_{n,i})&(\text{by \eqref{Fnmeasure}})\\
&\leq \sum_{1\le i\le p_n} 2 c_i^2I_s\left(E\cap B(x_{n,i}, 2^{-n-2})\right)
  &(\text{by \eqref{cube4}})\\
&=\sum_{1\le i\le p_n}\frac{2\mu(Q_{n,i})^2I_s\left(E\cap B(x_{n,i}, 2^{-n-2})\right)}
  {\L\bigl(E\cap B(x_{n,i}, 2^{-n-2})\bigr)^2}&(\text{by \eqref{cube0}})\\
&\leq  \sum_{1\le i\le p_n} (2k^2\alpha^{-2}\beta 8^s)  \mu(Q_{n,i})^2 2^{(n-1)s}
  &(\text{by \eqref{changevar} and \eqref{lemma-5-assump}})\\
&\leq (2k^2\alpha^{-2}\beta 8^s)\sum_{1\le i\le p_n}
  \iint_{ Q_{n,i}\times Q_{n,i}} |x-y|^{-s}\;d\mu(x)d\mu(y)
  &(\text{by \eqref{cube2}})\\
&\leq (2k^2\alpha^{-2}\beta 8^s)\iint_{\{(x,y)\;:\; |x-y|\leq 2\cdot 2^{-n}\}}
  |x-y|^{-s} \;d\mu(x)d\mu(y)\\
&\leq\tfrac\varepsilon4 I_s(\mu). &(\text{by \eqref{e-T1}})
\end{aligned}
\end{equation*}
We conclude that $(I)+(II)\leq\tfrac\varepsilon2I_s(\mu)$, from which
\eqref{e-u6} follows. This completes the proof of Lemma \ref{lemma-5}.
\end{proof}

Now we are ready to state the main technical result of this section.

\begin{proposition}\label{thm:energy_lemma}
Let $E\subset\R^d$ be a bounded Borel set with $\mathcal L(E)>0$, and
let $0<s<d$. Assume that $\mu\in\CP(E^+)$ with $I_s(\mu)<\infty$.
Then there is a sequence $(F_n)_{n\in\N}$ of Borel subsets of $E^+$
with positive  Lebesgue measure such that
the sequence $\mu_n:=\mathcal L(F_n)^{-1}\mathcal L\vert_{F_n}$, $n\in\N$,
converges to $\mu$ in the weak-star topology as $n$ tends to infinity, and
$\lim_{n\to\infty}I_s(\mu_n)=I_s(\mu)$.
\end{proposition}

\begin{proof}
Without loss of generality, we may assume that $E\subset\overline B(0,1)$. For
every $k,m\in\N$, define
\[
E_{k,m}:=\bigl\{x\in E:\frac{\L(E\cap B(x,r))}{\L(B(x,r))}>\frac{1}{k}
\text{ for all }0<r\leq 2^{-m}\bigr\},
\]
and set $E_k:=\bigcup_{m=1}^\infty E_{k,m}$. Then $E_{k,m}\uparrow E_k$ as $m$
tends to
infinity and, moreover, $E_k\uparrow E^+$ as $k$ tends to infinity.
Choose sufficiently large $k_0\in\N$ such that $\mu(E_{k_0})>0$.
For every integer $k\geq k_0$, pick $m_k\in\mathbb N$ such that
\[
\mu(E_{k, m_k})\geq (1-\tfrac 1k)\mu(E_k).
\]
Since $E_k\uparrow E^+$ as $k$ tends to infinity and $\mu$ is supported on
$E^+$, we have
\[
\lim_{k\to \infty} \mu(E_{k, m_k})=1.
\]
Set $\eta_k:=\mu(E_{k, m_k})^{-1}\mu|_{E_{k, m_k}}$ for all $k\geq k_0$. From
Lemma~\ref{lemma-3}, we obtain
\begin{equation}\label{e-convergence}
\lim_{k\to \infty} \eta_k=\mu \quad \text{ and }\quad
\lim_{k\to \infty} I_s(\eta_k)=I_s(\mu).
\end{equation}

Let $k\geq k_0$. Replacing in Lemma~\ref{lemma-5} the sets $E$ and $E_0$
by $E^+$ and $E_{k, m_k}$, respectively, implies the existence of a sequence
$(F_{k,i})_{i\in\N}$ of Borel subsets of $E^+$ such that
\[
\lim_{i\to \infty}\L(F_{k,i})^{-1}\L|_{F_{k,i}}=\eta_k \quad \text{ and }\quad
\lim_{i\to \infty} I_s\bigl(\L(F_{k,i})^{-1}\L|_{F_{k,i}}\bigr)=I_s(\eta_k).
\]
Combining this with \eqref{e-convergence}, we see that there exists a sequence
$(i_k)_{k\in\N}$ of natural numbers such that
\[
\lim_{k\to \infty}\L(F_{k,i_k})^{-1}\L|_{F_{k,i_k}}=\mu \quad \text{ and }\quad
\lim_{k\to \infty} I_s\bigl(\L(F_{k,i_k})^{-1}\L|_{F_{k,i_k}}\bigr)=I_s(\mu).
\]
This completes the proof of Proposition \ref{thm:energy_lemma}.
\end{proof}

Next lemma states that every Lebesgue measurable set with positive Lebesgue
density is contained in a Borel set with positive Lebesgue density having the
same Hausdorff content as the original set.

\begin{lemma}\label{lemma-Borel}
Let $R>0$ and $s>0$. Assume that $E\subset B(0,R)$ is a Lebesgue measurable
subset of $\mathbb R^d$. Then there
exists a Borel set $X\subset B(0,R)$ such that $E\subset X$,
$\L(X\setminus E)=0$ and
$\mathcal H^s_\infty(E)=\mathcal H^s_\infty(X)$. Furthermore, under the
additional assumption $E^+=E$, we may choose $X$ so that $X^+=X$.
\end{lemma}

\begin{proof}
The definition of $\mathcal H^s_\infty(\cdot)$ implies that for every $n\in\N$,
there exists a sequence $(F_{n,i})_{i\in\N}$ of Borel sets satisfying
$E\subset\bigcup_{i=1}^\infty F_{n,i}$ and
$\sum_{i=1}^\infty(\diam F_{n,i})^s<\mathcal H^s_\infty(E)+\frac 1n$. Defining
$F:=\bigcap_{n=1}^\infty \bigcup_{i=1}^\infty F_{n,i}$, it is clear that
$F$ is Borel measurable, $E\subset F$ and
$\mathcal H^s_\infty(F)=\mathcal H^s_\infty(E)$.
Moreover, there exists a Borel set $A\subset B(0,R)$ such that $E\subset A$
and $\L(A)=\L(E)$. Setting $X:=F\cap A$, it is easy to see that $X$ fulfils
all the desired properties.

If $E^+=E$, the above construction may lead to the situation where $X^+\ne X$.
However, we have $E\subset X^+\subset X$ and, therefore,
$\mathcal H^s_\infty(X^+)=\mathcal H^s_\infty(E)=\mathcal H^s_\infty(X)$ and
$\L(X^+)=\L(E)=\L(X)$. Note that $(X^+)^+=X^+$. (Indeed, $(A^+)^+=A^+$
for all Lebesgue measurable sets $A\subset\R^d$ since $\L(A^+)=\L(A)$.) Since
$X$ is a Borel set, so is $X^+$. We complete the proof by deducing that the set
$Y:=X^+$ has the following properties:
$Y\subset B(0,R)$ is a Borel set, $Y^+=Y$, $E\subset Y$, $\L(Y\setminus E)=0$
and $\mathcal H^s_\infty(Y)=\mathcal H^s_\infty(E)$.
\end{proof}

The next lemma may be regarded as a complementary result to Lemma~\ref{HgeG}.

\begin{lemma}\label{HleG}
Let $0<t<s<d$ and $R>0$. Then there exists a positive constant
$C=C(s,t,d,R)$ such that for all Lebesgue measurable sets
$E\subset B(0,R)$ with $E^+=E$, we have
\[
\mathcal H_\infty^s(E)\le CG_t(E).
\]
\end{lemma}

\begin{proof}
We may assume that $E\ne\emptyset$. Since $E^+=E\ne\emptyset$, we have
$\mathcal L(E)>0$ and, therefore, $\mathcal H_\infty^s(E)>0$.

We first assume that $E$ is Borel measurable. By Frostman's lemma
\cite[Theorem 8.8]{Mat}, there exists a Radon measure
$\mu$ supported on $E$ such that $\mu(B(x,r))\le r^s$ for all $x\in\mathbb R^d$
and $r>0$ and, moreover, $\mu(E)=c\mathcal H_\infty^s(E)$, where $c$ is
a constant depending only on $d$. Next we make a standard calculation in a
slightly complicated looking fashion since that will be useful for later
purposes (see Section~\ref{finalremarks}). Let $h(r):=r^t$ and
$\tilde h(r):=r^s$ for all
$r\ge 0$. Set $\delta:=\frac st-1$ and $a:=\mu(E)^{-\frac 1{1+\delta}}$. Then
$\tilde h(r)\le h(r)^{1+\delta}$ for all $r>0$. By \cite[Theorem 1.15]{Mat}, we
have for some constant $c_1$ depending on $t$ and $s$ that
\begin{align*}
0&<I_t(\mu)=\iint\mu\bigl(\{y\in\R^d:\frac 1{h(|x-y|)}\ge u\}\bigr)
  \,du\,d\mu(x)\\
&\le\iint\min\{\mu(E),\mu(B(x,h^{-1}(u^{-1})))\}\,du\,d\mu(x)\\
&\le\int\Bigl(\int_0^a\mu(E)\,du+\int_a^\infty \tilde h(h^{-1}(u^{-1}))\,du
  \Bigr)\,d\mu(x)\\
&\le\int\bigl(\mu(E)^{1-\frac 1{1+\delta}}+\int_a^\infty u^{-1-\delta}\,du\bigr)\,
  d\mu(x)\\
&\le c_1\mu(E)^{1+\frac\delta{1+\delta}}
  \le c_1\bigl(c\mathcal H_\infty^s(B(0,R))\bigr)^{1-\frac ts}\mu(E).
\end{align*}
Thus $0<I_t(\mu)\le \tilde c\mu(E)<\infty$, where $\tilde c$ is a constant
depending only on $t$, $s$ and $R$.

Applying Proposition~\ref{thm:energy_lemma} to $E$, we find a sequence
$(\mu_k)_{k\in\N}$ of measures such that $\mu_k=\mu(E)\L(E_k)^{-1}\L|_{E_k}$ and
$\lim_{k\to\infty}I_t(\mu_k)=I_t(\mu)$. Here each $E_k$ is a
Borel measurable subset of $E$ with $0<\mathcal L(E_k)<\infty$.
For all sufficiently large $k\in\N$, we obtain
\[
\frac{\mu(E)^2}{\L(E_k)^2}I_t(E_k)= I_t(\mu_k)\le2I_t(\mu)
\le 2\tilde c\mu(E),
\]
giving
\[
\mathcal H_\infty^s(E)=c^{-1}\mu(E)\le 2\tilde cc^{-1}\frac{\L(E_k)^2}{I_t(E_k)}
  \le 2\tilde cc^{-1}G_t(E)
\]
by \eqref{defG}. Choosing $C:=2\tilde cc^{-1}$, completes the proof for Borel
sets $E$.

The general case of $E$ being Lebesgue measurable may be reduced to the above
setting in the following manner. By Lemma~\ref{lemma-Borel}, there exists
a Borel set $X\subset B(0,R)$
so that $X^+=X$, $E\subset X$, $\mathcal H^s_\infty(X)=\mathcal H^s_\infty(E)$ and
$\L(X\setminus E)=0$.  Since $E\subset X$ and $\L(X\setminus E)=0$, we have
$G_t(X)=G_t(E)$.  The above established inequality
$\mathcal H_\infty^s(X)\le CG_t(X)$ implies that
$\mathcal H_\infty^s(E)\le CG_t(E)$.
\end{proof}

As an immediate consequence of Lemma \ref{HleG} we prove that the quantities
$t_0(\bA)$ and $s_0(\bA)$ defined in \eqref{deft0} and \eqref{defs0},
respectively, agree provided that the sets $A_n$ are Lebesgue measurable,
uniformly bounded and $A_n^+=A_n$.

\begin{corollary}\label{H=G}
Let $R>0$. Assume that $\bA:=(A_n)_{n\in\N}$ is a sequence of Lebesgue measurable
subsets of $B(0,R)\subset\mathbb R^d$ such that $A_n^+=A_n$ for all
$n\in\mathbb N$. Then $t_0(\bA)=s_0(\bA)$.
\end{corollary}

\begin{proof}
The claim follows immediately from Lemmas \ref{HgeG} and \ref{HleG}.
\end{proof}

\begin{proof}[Proof of Theorem \ref{main}.(c)]
The statement follows directly from Theorem \ref{main}.(a)-(b) and Corollary \ref{H=G}, where the proof of Theorem \ref{main}.(b) will be given in Section~\ref{lowerboundsection}.
\end{proof}

\section{Minimal regular energy}\label{regularenergy}

In this section, we introduce a new concept of minimal regular energy and study
basic properties of it. We also
explain how it can be used to estimate dimensions of random covering sets. The main results are Proposition~\ref{pro-1.5} and Lemma~\ref{lem-6}, which  are needed in our proof of Theorem~\ref{main}.(b).

For $E\subset\R^d$, set
\[
\CP_0(E):=\bigl\{\mu\in\CP(E):\mu=\sum_{i=1}^k c_i\L|_{E_i}\text{ where }k\in\N,
               c_i>0\text{ and }E_i\subset E
\text{ are Borel sets}\bigr\}.
\]
Recall from Section~\ref{upperbound} that $\CP(E)$ is the space of Borel
probability measures supported on $E$.
For $E\subset \R^d$ and $0<s<d$, define
\[
\Gamma_s(E):=\left\{\begin{array}{ll} \inf\{I_s(\mu):\mu\in\CP_0(E)\},&
                                                     \text{ if } \L(E)>0,\\
                                       \infty,&\text{ if }\L(E)=0.
\end{array}\right.
\]
The quantity $\Gamma_s(E)$ is called the
{\it minimal regular $s$-energy of $E$}.

\begin{lemma}\label{lem-1.2}
Let $E\in \B(\R^d)$ and $0<s<d$. Then the following properties hold:
\begin{itemize}
\item[(i)] If $F\subset E$ is a Borel set, then $\Gamma_s(E)\leq \Gamma_s(F)$.
\item[(ii)] If $\L(E)>0$, then $\Gamma_s(E)<\infty$.
\item[(iii)] If $E$ is bounded, then $\Gamma_s(E)>0$.
\item[(iv)] For every $\varepsilon>0$, there exists
$\delta=\delta(E,\varepsilon)>0$ such that
\[
\Gamma_s(F)\leq \Gamma_s(E)+\varepsilon
\]
provided that $F\in\B(\R^d)$ and $\L(E\setminus F)<\delta$.
\item[(v)] Let $(E_n)_{n\in\N}$ be a sequence of Borel subsets of $E$.
Supposing that $\L(E)<\infty$, we have
\[
\Gamma_s\bigl(\bigcup_{i=1}^\infty E_i\bigr)=\lim_{n\to\infty}\Gamma_s
  \bigl(\bigcup_{i=1}^n E_i\bigr)
=\inf_{n\in\mathbb N}\Gamma_s\bigl(\bigcup_{i=1}^n E_i\bigr).
\]
\end{itemize}
\end{lemma}

\begin{proof}
The statement (i) is obvious. To verify (ii), choose a compact set
$F\subset E$ with $\L(F)>0$ (recall that $\L$ is inner regular), and set
$\mu:=\L(F)^{-1}\L|_F$. Clearly, $\mu\in \CP_0(E)$. Since $0<s<d$, we have
\begin{align*}
 I_s(\mu) &=\L(F)^{-2}\iint_{F\times F} |x-y|^{-s}\;d\L(x)d\L(y)\\
 &\leq\L(F)^{-2}\iint_{B(0,R)\times B(0, R)} |x-y|^{-s}\;d\L(x)d\L(y)<\infty,
\end{align*}
where $R>0$ is sufficiently large so that $F\subset B(0,R)$. Hence
$\Gamma_s(E)<\infty$.

For the purpose of proving (iii), suppose on the contrary that
$\Gamma_s(E)=0$. Then there exists
a sequence $(\mu_n)_{n\in\N}$ such that $\mu_n\in \CP_0(E)$ and
$\lim_{n\to \infty} I_s(\mu_n)=0$.  Since $E$ is bounded, the sequence
$(\mu_n)_{n\in\N}$
has at least one accumulation point, say $\mu$, in the weak-star topology. By
Lemma \ref{lemma-1}, $I_s(\cdot)$ is lower semi-continuous and, therefore,
$I_s(\mu)=0$, leading to a contradiction since $\mu$ is a Borel probability
measure. This completes the proof of (iii).

Next we verify (iv). We may assume that $\L(E)>0$. Then
$\Gamma_s(E)<\infty$ by (ii). Let $\varepsilon>0$. By the definition of
$\Gamma_s(\cdot)$, there
exists $\mu=\sum_{i=1}^kc_i \L|_{E_i}\in \CP_0(E)$ such that
$I_s(\mu)\leq\Gamma_s(E)+\tfrac{\varepsilon}2$
where $E_i\subset E$ and $\L(E_i)>0$ for all $i=1,\dots,k$. Define
$\delta_1:=\min\{\L(E_i):i=1,\dots,k\}$,
\[
\gamma:=\sqrt{\frac{\Gamma_s(E)+\tfrac{\varepsilon}2}{\Gamma_s(E)+\varepsilon}}
\quad\text{ and }\quad\delta:=(1-\gamma)\delta_1.
\]
Let $F\in\B(\R^d)$ with
$\L(E\setminus F)<\delta$. We proceed by  showing  that
$\Gamma_s(F)\leq\Gamma_s(E)+\epsilon$.  First notice that
for all $i=1,\dots,k$,
\begin{eqnarray*}
\L(E_i\cap F)&\geq& \L(E_i\cap E)-\L(E\setminus F)
= \L(E_i)-\L(E\setminus F)\\
&\geq&   \L(E_i)-\delta\geq \gamma \L(E_i).
\end{eqnarray*}
Letting $0<\varrho<1$, by the inner regularity of $\L$, there is a compact set
$\widetilde E_i\subset E_i\cap F$ such that
$\L(\widetilde E_i)\ge\varrho\L(E_i\cap F)$ for all $i=1,\dots,k$. Setting
\[
\mu_F:=\frac 1{c_F}\sum_{i=1}^kc_i\L|_{\widetilde E_i},
\]
where $c_F:=\sum_{i=1}^k c_i\L(\widetilde E_i)>0$, the measure $\mu_F$ is
supported on
$F$ and, therefore, $\mu_F\in\CP_0(F)$.
Using the fact that $E_i\subset E$ for all $i=1,\dots,k$, we deduce that
\[
\L(\widetilde E_i)\ge\varrho\L(E_i\cap F)>\varrho\gamma\L(E_i).
\]
 Thus,
$c_F\geq\sum_{i=1}^k c_i\varrho\gamma\L(E_i)=\varrho\gamma$
and
\[
I_s(\mu_F)\leq (c_F)^{-2}I_s(\mu)\leq \rho^{-2}\gamma^{-2} (\Gamma_s(E)
 +\tfrac{\varepsilon}2)=\varrho^{-2}(\Gamma_s(E)+\varepsilon).
\]
Hence $\Gamma_s(F)\leq\varrho^{-2}(\Gamma_s(E)+\varepsilon).$
Letting $\varrho$ tend to 1, gives
$\Gamma_s(F)\leq\Gamma_s(E)+\varepsilon$.
This completes the proof of (iv).

Finally, (v) follows from (i), (iv) and the fact that
\[
\lim_{n\to\infty}\L\Bigl(\bigcup_{i=1}^\infty E_i\setminus\bigl(\bigcup_{j=1}^n E_j
  \bigr)\Bigr)=0.
\]
\end{proof}

We proceed by giving an equivalent definition of $\Gamma_s(A)$
although we will not apply it in this paper. We use the notation
$\mu\ll\nu$ to indicate that the measure $\mu$ is absolutely
continuous with respect to the measure $\nu$.

\begin{lemma}\label{lem-1.1}
Let $E\subset\R^d$. With convention $\inf\emptyset=\infty$, we have
\[
\Gamma_s(E)=\inf\{I_s(\mu):\mu\in\CP(E)\text{ with }\mu\ll\L\}.
\]
\end{lemma}

\begin{proof}
It is sufficient to show that for every $\mu\in\CP(E)$ with $\mu\ll \L$ and for
every $\varepsilon>0$, there exists $\eta\in\CP_0(E)$ such that
\[
I_s(\eta)\leq I_s(\mu)+\varepsilon.
\]
To prove the above fact, we denote by $h=\frac{d\mu}{d\L}$ the Radon-Nikodym
derivative of $\mu$ with respect to $\L$.
Approximating $h$ by step functions, we see that for every $\delta>0$, there
exists a step function $g=\sum_{i=1}^k a_i\chi_{E_i}$, where $a_i>0$, $E_i$ is a
Borel set and $\overline E_i\subset E$ for all $i=1,\dots,k$, so that
\begin{equation}\label{e-7'}
 \bigl|I_s(\mu)-\iint |x-y|^{-s}g(x)g(y)\;d\L(x)d\L(y)\bigr|<\frac{\varepsilon}2
\end{equation}
and
\begin{equation}\label{e-7}
 \bigl|\int h(x)\; d\L(x)-\int g(x)\; d\L(x)\bigr|<\delta.
\end{equation}
Let $u:=\sum_{i=1}^k a_i\L(E_i)$. Then $u>1-\delta$ by \eqref{e-7}. Defining
$\eta:=\frac 1u\sum_{i=1}^k a_i\L|_{E_i}$, implies that $\eta\in\CP_0(E)$. Using
\eqref{e-7'}, we get for a small enough $\delta$ that
\begin{align*}
 I_s(\eta)&= u^{-2} \iint |x-y|^{-s} g(x)g(y)\; d\L(x)d\L(y)\\
   &\leq(1-\delta)^{-2}(I_s(\mu)+\tfrac{\varepsilon}2)\\
   &\leq I_s(\mu)+\varepsilon.
\end{align*}
This completes the proof of the lemma.
\end{proof}

\begin{lemma}\label{lem-1.0}
Let $(E_n)_{n\in\N}$ be a decreasing sequence of compact subsets of $\R^d$, and
let $0<s<d$. Assume that there
exists $c>0$ such that $\Gamma_s(E_n)<c$ for all $n\in\N$. Then
$\mathcal H^s_\infty\big(\bigcap_{n=1}^\infty E_n\big)\geq c^{-1}$. In
particular, $\dimH\big(\bigcap_{n=1}^\infty E_n\big)\geq s$.
\end{lemma}

\begin{proof}
According to the definition of $\Gamma_s(\cdot)$, for every $n\in \N$,  there
exists $\mu_n\in\CP_0(E_n)$ so that $I_s(\mu_n)<c$. Let $\mu$ be an accumulation
point of the sequence $(\mu_n)_{n\in\N}$ in the weak-star topology. Then $\mu$ is
supported on $\bigcap_{n=1}^\infty E_n$ and, furthermore, $I_s(\mu)\leq c$
by lower
semi-continuity of $I_s(\cdot)$ (see Lemma~\ref{lemma-1}). The
conclusions follow from Remark~\ref{contentbound}.
\end{proof}

In the remaining part of this section, we assume that $U\subset\R^d$ is open
and $(A_n(x))_{n\in\N}$ is a sequence of $\B(\R^d)$-valued functions defined
on $U$ such that
\begin{itemize}
\item[(C-1)] $\L(A_n(x))<\infty$ for all $x\in U$ and $n\in \N$, and
\item[(C-2)] $\lim_{y\to x}\L\bigl((A_n(x)\setminus A_n(y))\cup(A_n(y)
             \setminus A_n(x))\bigr)=0$ for all $x\in U$ and $n\in\N$.
\end{itemize}
Let $U^\N:=\prod_{n=1}^\infty U$ be endowed with the product topology. Consider
$\eta\in\CP(U)$ and set $\P:=\prod_{i=1}^\infty \eta$.

\begin{lemma}\label{lem-4}
Let $E\in \B(\R^d)$ with $\L(E)<\infty$.
Then, for all $n\in\N$, the mapping
\[
(x_i)_{i=1}^n\mapsto\Gamma_s\bigl(E\cap\bigcup_{i=1}^n A_i(x_i)\bigr)
\]
is upper semi-continuous on $U^n:=\prod_{i=1}^n U$. Moreover, the mapping
\[
\x:=(x_i)_{i=1}^\infty\mapsto\Gamma_s\bigl(E\cap\bigcup_{i=1}^\infty A_i(x_i)\bigr)
\]
is Borel measurable on $U^\N$.
\end{lemma}

\begin{proof}
Let $\x\in U^\N$ and $n\in\N$. By (C-2),
$\L\bigl(\bigcup_{i=1}^n A_i(x_i)\setminus\bigcup_{j=1}^n A_j(y_j)\bigr)$
is close to 0 when
$(y_i)_{i=1}^n \in U^n$
is close to $(x_i)_{i=1}^n$. Applying Lemma \ref{lem-1.2}.(iv), we obtain
upper
semi-continuity (and hence Borel measurability) of the mapping
$(x_i)_{i=1}^n\mapsto\Gamma_s\bigl(E\cap\bigcup_{i=1}^n A_i(x_i)\bigr)$ defined on
$U^n$ and that of the mapping
$\x\mapsto\Gamma_s\bigl(E\cap\bigcup_{i=1}^n A_i(x_i)\bigr)$ defined on $U^\N$.
It follows from Lemma~\ref{lem-1.2}.(v) that
\[
\lim_{n\to\infty}\Gamma_s\bigl(E\cap\bigcup_{i=1}^n A_i(x_i)\bigr)
  =\Gamma_s\bigl(E\cap\bigcup_{i=1}^\infty A_i(x_i)\bigr),
\]
which, in turn, implies Borel measurability of the map
$\x\mapsto\Gamma_s\bigl(E\cap\bigcup_{i=1}^\infty A_i(x_i)\bigr)$ on $U^\N$.
\end{proof}

Next proposition provides a sufficient condition for determining a lower
bound for Hausdorff dimensions of typical random covering sets.

\begin{proposition}\label{pro-1.5}
Let $E\subset\R^d$ be compact with $\L(E)>0$. In addition to (C-1)
and (C-2), assume that $A_n(x)$ is compact for all $x\in U$ and $n\in\N$. Let
$0<s<d$. Suppose that for all compact sets $F\subset E$, we have for
$\P$-almost all $\x\in U^\N$ that
\begin{equation}\label{e-5}
\Gamma_s\Bigl(F\cap\bigl(\bigcup_{i=n}^\infty A_i(x_i)\bigr)\Bigr)=\Gamma_s(F)
  \text{ for all }n\in\N.
\end{equation}
Then
\[
\mathcal H_\infty^s\bigl(\limsup_{n\to \infty} A_n(x_n)\bigr)\ge\Gamma_s(E)^{-1}
\text{ and }\dimH\bigl(\limsup_{n\to \infty} A_n(x_n)\bigr)\geq s
\]
for $\P$-almost all $\x\in U^\N$.
\end{proposition}

\begin{proof}
From \eqref{e-5}, we obtain
\[
\Gamma_s\Bigl(E\cap\bigl(\bigcup_{i=1}^\infty A_i(x_i)\bigr)\Bigr)=\Gamma_s(E)
\]
for $\P$-almost all $\x\in U^\N$. Note that $0<\Gamma_s(E)<\infty$ by
Lemma~\ref{lem-1.2} claims (ii) and (iii). Letting $\ell>2$,
Lemma~\ref{lem-1.2}.(v) and
Lemma~\ref{lem-4} imply the existence of a Borel measurable function
$n_1:U^\N\to\N$ such that
\[
\Gamma_s\Bigl(E\cap\bigl(\bigcup_{i=1}^{n_1(\x)} A_i(x_i)\bigr)\Bigr)
  <(1+\ell^{-1})\Gamma_s(E)
\]
for $\P$-almost all $\x\in U^\N$. By Lemma~\ref{lem-1.2}.(i), we find
$N_1\in\mathbb N$ and a Borel set $\Lambda_1\subset U^{N_1}$ so that
\begin{equation}\label{Gammasineq}
\eta^{N_1}(\Lambda_1)>1-\ell^{-1}\quad\text{ and }\quad\Gamma_s
 \Bigl(E\cap\bigl(\bigcup_{i=1}^{N_1} A_i(x_i)\bigr)\Bigr)<(1+\ell^{-1})\Gamma_s(E)
\end{equation}
for all $(x_1,\ldots,x_{N_1})\in\Lambda_1$, where
$\eta^{N_1}:=\prod_{i=1}^{N_1}\eta$. Applying \eqref{e-5} with
$F=E\cap\bigl(\bigcup_{i=1}^{N_1} A_i(x_i)\bigr)$, gives for all
$(x_1,\ldots,x_{N_1})\in\Lambda_1$ that
\begin{equation}\label{e-F}
\Gamma_s\Bigl(E\cap\bigl(\bigcup_{i=1}^{N_1} A_i(x_i)\bigr)\cap
  \bigl(\bigcup_{j=N_1+1}^\infty A_j(x_j)\bigr)\Bigr)<(1+\ell^{-1})\Gamma_s(E)
\end{equation}
for $\bigl(\prod_{i=N_1+1}^\infty\eta\bigr)$-almost all
$(x_{N_1+1},x_{N_1+2},\ldots)\in\prod_{i=N_1+1}^\infty U$.
Moreover, by Fubini's theorem, inequality \eqref{e-F} holds for $\P$-almost all
$\x\in\Lambda_1\times\prod_{i=N_1+1}^\infty U$. As above, it follows from
Lemma~\ref{lem-1.2}.(i) that there exist a natural
number $N_2>N_1$ and a Borel set
$\Lambda_2\subset\Lambda_1\times\prod_{i=N_1+1}^{N_2} U\subset U^{N_2}$ with
$\eta^{N_2}(\Lambda_2)\geq\eta^{N_1}(\Lambda_1)-\ell^{-2}$ such that
\[
\Gamma_s\Bigl(E\cap\bigl(\bigcup_{i=1}^{N_1} A_i(x_i)\bigr)\cap
 \bigl(\bigcup_{j=N_1+1}^{N_2} A_j(x_j)\bigr)\Bigr)
 <(1+\ell^{-1})(1+\ell^{-2})\Gamma_s(E)
\]
for all $(x_1,\dots,x_{N_2})\in\Lambda_2$.

By induction, we deduce that there exist a strictly increasing sequence
$(N_n)_{n\in\N}$ of positive integers and a sequence $(\Lambda_n)_{n\in\mathbb N}$
of Borel sets such that $\Lambda_n\subset U^{N_n}$,
$\Lambda_{n+1}\subset\Lambda_{n}\times\prod_{i=N_n+1}^{N_{n+1}} U$,
\begin{equation}\label{etae-1}
\eta^{N_{n+1}}(\Lambda_{n+1})\geq \eta^{N_n} (\Lambda_n)-\ell^{-n-1}
\end{equation}
and
\begin{equation}\label{e-2}
\Gamma_s\Bigl(E\cap\bigcap_{k=1}^n\bigl(\bigcup_{i=N_{k-1}+1}^{N_{k}} A_i(x_i)\bigr)
  \Bigr)<\bigl(\prod_{i=1}^n(1+\ell^{-i})\bigr)\Gamma_s(E)
\end{equation}
for all $(x_1,\ldots,x_{N_n})\in\Lambda_n$. Here $N_0:=0$.
Defining $\Omega:=\bigcap_{n=1}^\infty (\Lambda_n\times\prod_{i=N_1+1}^\infty U)$
and using \eqref{etae-1}, implies that
\begin{equation}\label{e-3}
 \P(\Omega)\geq 1-\sum_{n=1}^\infty \ell^{-n}=\frac{\ell-2}{\ell-1}.
\end{equation}
Moreover, by \eqref{e-2} and Lemma~\ref{lem-1.0}, we have
\begin{equation}\label{midresult}
\begin{split}
 \mathcal H_\infty^s\bigl(\bigcap_{n=1}^\infty\bigcup_{i=N_{n-1}+1}^{N_n} A_i(x_i)\bigr)
&\ge\Bigl(\bigl(\prod_{i=1}^\infty(1+\ell^{-i})\bigr)\Gamma_s(E)\Bigr)^{-1}
 \text{ and }\\
 \dimH\bigl(\bigcap_{n=1}^\infty\bigcup_{i=N_{n-1}+1}^{N_n} A_i(x_i)\bigr)
&\geq s
\end{split}
\end{equation}
for all $\x\in\Omega$. This gives
$\dimH\bigl(\limsup_{n\to \infty} A_n(x_n)\bigr)\geq s$ for all $\x\in\Omega$.
Since $\ell$ can be taken arbitrarily large, it follows from \eqref{e-3} that
\[
\mathcal H_\infty^s\bigl(\limsup_{n\to \infty} A_n(x_n)\bigr)\ge\Gamma_s(E)^{-1}
\text{ and }\dimH\bigl(\limsup_{n\to \infty} A_n(x_n)\bigr)\geq s
\]
for $\P$-almost all $\x\in U^\N$.
\end{proof}

The above proof readily leads to the following deterministic version of
Proposition~\ref{pro-1.5}, which may be of independent interest.

\begin{proposition}\label{pro-1.5det}
Let $E\subset\R^d$ be compact with $\L(E)>0$, and let $(A_n)_{n\in\N}$ be a
sequence of compact  subsets of $\R^d$. Let $0<s<d$. Suppose that for all
compact sets $F\subset E$, we have that
\begin{equation*}
\Gamma_s\Bigl(F\cap\bigl(\bigcup_{i=n}^\infty A_i\bigr)\Bigr)=\Gamma_s(F)
  \text{ for all }n\in\N.
\end{equation*}
Then
\[
\mathcal H_\infty^s\bigl(\limsup_{n\to\infty} A_n\bigr)\ge\Gamma_s(E)^{-1}
\text{ and }\dimH\bigl(\limsup_{n\to\infty} A_n\bigr)\geq s.
\]
\end{proposition}

In the last result of this section, we give a sufficient condition for the
validity of \eqref{e-5}. Recall that, by the definition of $\Gamma_s(\cdot)$,
the inequality \eqref{e-5} is valid for all $F\in\B(\R^d)$ with $\L(F)=0$.

\begin{lemma}\label{lem-6}
Let $F\in\B(\R^d)$ with $\L(F)>0$, and let $0<s<d$. Assume that for every
$\varepsilon>0$ and $\delta>0$ and for every $n\in\N$, there exist an integer
$N>n$ and a Borel
measurable set $\Omega\subset U^\N$ with $\P(\Omega)>1-\delta$ such that
\begin{equation}\label{e-8'}
\int_\Omega\Gamma_s\bigl(F\cap\bigcup_{i=n}^N A_i(x_i)\bigr)\; d\P(\x)
  <\Gamma_s(F)+\varepsilon.
\end{equation}
Then for $\P$-almost all $\x\in U^\N$,
\[
\Gamma_s\bigl(F\cap\bigcup_{i=n}^\infty A_i(x_i)\bigr)=\Gamma_s(F)
\]
for all $n\in\N$.
\end{lemma}

\begin{proof}
Let $n\in\N$ and $\gamma>0$. By Lemma~\ref{lem-1.2}.(i),
\[
\Gamma_s\bigl(F\cap\bigcup_{i=n}^N A_i(x_i)\bigr)\geq\Gamma_s\bigl(F\cap
  \bigcup_{i=n}^\infty A_i(x_i)\bigr)\geq\Gamma_s(F)
\]
for all $\x\in U^\N$ and $N\in\N$. Let
\[
\Omega':=\Bigl\{x\in U^\N:\Gamma_s\Bigl(F\cap\bigl(\bigcup_{i=n}^\infty A_i(x_i)
  \bigr)\Bigr)\geq\Gamma_s(F)+\gamma\Bigr\}.
\]
It follows from Lemma~\ref{lem-4} that $\Omega'$ is a Borel set.
We show that $\P(\Omega')=0$. Suppose on the contrary that
$\P(\Omega')>0$ and choose
\begin{equation}\label{e-9'}
\varepsilon:=\frac{\P(\Omega')\gamma}2\quad\text{ and }\quad
  \delta:=\frac{\P(\Omega')\gamma}{2(\gamma+2\Gamma_s(F))}.
\end{equation}
Recall that $\Gamma_s(F)<\infty$ by Lemma~\ref{lem-1.2}.(ii). Then for all
integers $N>n$ and for all Borel measurable sets $\Omega\subset U^\N$ with
$\P(\Omega)>1-\delta$, we have
\begin{equation*}
\begin{split}
\int_\Omega &\Gamma_s\bigl(F\cap\bigcup_{i=n}^N A_i(x_i)\bigr)\; d\P(\x)\\
 &\geq\int_{\Omega\setminus\Omega'}\Gamma_s(F)\;d\P(\x)
   +\int_{\Omega\cap\Omega'} \Gamma_s(F)+\gamma\;d\P(\x)\\
 &=\P(\Omega\setminus\Omega')\Gamma_s(F)+\P(\Omega\cap \Omega')
   (\Gamma_s(F)+\gamma)\\
 &\geq(\P(\Omega)-\P(\Omega'))\Gamma_s(F)+(\P(\Omega')+\P(\Omega)-1)
   (\Gamma_s(F)+\gamma) \\
 &=(2\P(\Omega)-1)\Gamma_s(F)+(\P(\Omega')+\P(\Omega)-1)\gamma\\
 &\geq(1-2\delta)\Gamma_s(F)+(\P(\Omega')-\delta)\gamma
\qquad\qquad\qquad\qquad\qquad\qquad(\text{by \eqref{e-9'}})\\
 &=\Gamma_s(F)+\varepsilon.
\end{split}
\end{equation*}
This contradicts \eqref{e-8'} and completes the proof.
\end{proof}

\section{Lower bound for Hausdorff dimension}\label{lowerboundsection}

The main purpose of this section is to verify Theorem~\ref{main}.(b). This
is achieved by showing first that, under certain assumptions on the measure
$\eta\in\CP(U)$ and the sequence $(A_n(x))_{n\in\N}$, the assumption
\eqref{e-8'} of Lemma~\ref{lem-6} holds. Theorem~\ref{main}.(b)
then follows by applying Lemma~\ref{lem-6} and Proposition~\ref{pro-1.5}.

We start with a simple observation on independent random variables.

\begin{lemma}\label{largesum}
Let $(a_n)_{n\in\N}$ be a sequence of positive numbers such that
$\sum_{n=1}^\infty a_n=\infty$, and let $0<c<1$. Suppose that
$(\omega_n)_{n\in\N}$ is a sequence of independent random variables with
$\omega_n\in\{0\}\cup[a_n,\infty[$ and the probability $\P$ satisfies
\begin{equation}\label{assumption}
\P(\omega_n\ne 0)\ge c.
\end{equation}
Then for all $N\in\N$ and $C\ge 0$, we have
\[
\lim_{M\to\infty}\P\bigl(\sum_{n=N}^M\omega_n\ge C\bigr)=1.
\]
\end{lemma}

\begin{proof}
Define $b_n:=\min\{1,a_n\}$ for $n\in\N$. Then either $b_n=1$ for
infinitely many $n\in\N$, or $b_n=a_n$ for all sufficiently large $n\in\N$. In
both of these cases we have $\sum_{n=1}^\infty b_n=\infty$.

Let $N\in\N$ and $C\ge 0$. Defining
\[
\tilde\omega_n:=
  \begin{cases}
    0,&\text{ if }\omega_n=0\\
    1,&\text{ if }\omega_n\ne 0,
  \end{cases}
\]
gives
\begin{equation}\label{newvariable}
\omega_n\ge\tilde\omega_nb_n
\end{equation}
for all $n\in\N$. Set $X_k:=\sum_{n=N}^{N+k}\tilde\omega_n b_n$ and
$C_k:=\E(X_k)$ for $k\in\N$, where the expectation with respect to $\P$
is denoted by $\E$. Then \eqref{assumption}
and the fact that $\sum_{n=1}^\infty b_n=\infty$ imply
\begin{equation}\label{expect}
C_k\ge c\sum_{n=N}^{N+k}b_n\ge 2C
\end{equation}
for all large enough $k\in\N$. Applying Hoeffding's inequality (see
\cite[Theorem 2]{Ho63}) to the independent random variables
$(\tilde\omega_n b_n)_{n=N}^{N+k}$, taking into account that
$\sum_{n=N}^{N+k}(b_n)^2\le \sum_{n=N}^{N+k}b_n$ (since $0<b_n\leq 1$) and using
\eqref{expect}, we have for all sufficiently large $k\in\N$ that
\begin{align*}
\P\bigl(X_k<\frac{C_k}2\bigr)
 &\le\P\bigl(|X_k-\E(X_k)|\ge\frac12\E(X_k)\bigr)
   \le2\exp\Bigl(\frac{-(C_k)^2}{2\sum_{n=N}^{N+k}(b_n)^2}\Bigr)\\
 &\le2\exp\Bigl(\frac{-(C_k)^2}{2\sum_{n=N}^{N+k}b_{n}}\Bigr)
   \le 2\exp\Bigl(-\frac12 c^2\sum_{n=N}^{N+k}b_n\Bigr).
\end{align*}
Finally, \eqref{newvariable}, \eqref{expect} and the above inequality combine
to give that for all large enough $k\in\N$
\begin{align*}
\P\Bigl(\sum_{n=N}^{N+k}\omega_n\ge C\Bigr)
 &\ge\P\Bigl(\sum_{n=N}^{N+k}\tilde\omega_nb_n\ge C\Bigr)
   \ge\P\bigl(X_k\ge\frac{C_k}2\bigr)\\
 &\ge 1-2\exp\Bigl(-\frac 12 c^2\sum_{n=N}^{N+k}b_n\Bigr)
   \xrightarrow[k\to\infty]{}1.
\end{align*}
This completes the proof.
\end{proof}

In the remaining part of this section, let $U\subset\R^d$ be open and let $E\subset U$ be a compact set with
$\L(E)>0$. Assume that $\eta\in\CP(U)$ satisfies $\eta(E)>0$,
$\eta|_E\ll\L|_E$
and
\begin{equation}\label{e-10}
\sup_{x,y\in E} \frac{h(x)}{h(y)}<\infty,
\end{equation}
where $h:=\frac{d\eta|_E}{d\L}$ is the Radon-Nikodym derivative of $\eta|_E$
with respect to $\L$. Set $\P:=\prod_{i=1}^\infty \eta$. Let $0<s<d$. Next we
define the sequence $(A_n(x))_{n\in\N}$.

\begin{definition}\label{sequenceAn}
Let  $y_0\in\R^d$. Assume that $(K_n)_{n\in\N}$ is a sequence of
compact sets in $\R^d$ satisfying
\begin{itemize}
\item[(i)] $\L(K_n)>0$,
\item[(ii)] $\lim_{n\to\infty}\diam K_n=0$,
\item[(iii)] $\lim_{n\to \infty}\dist(y_0,K_n)=0$\text{ and}
\item[(iv)] $\sum_{n=1}^\infty g_s(K_n)=\infty$ (recall \eqref{gtdef}).
\end{itemize}
Choose $r_0>0$ such that $K_n\subset B(y_0,r_0)$ for all $n\in\N$.
Assume that $W:U\times B(y_0,r_0)\to\R^d$ is a uniform
bidiffeomorphism (recall Definition~\ref{bidiffeo}) satisfying $W(x,y_0)=x$
for
all $x\in U$. Define $A_n(x):=W(x,K_n)$ for all $x\in U$ and $n\in\N$.
\end{definition}

The sequence $(A_n(x))_{n\in\N}$ has following properties:

\begin{lemma}\label{propertiesC}
Let $(A_n(x))_{n\in\N}$ be as in Definition~\ref{sequenceAn}. Then
the properties
(C-1) and (C-2) from Section~\ref{regularenergy} are satisfied. Furthermore,
\begin{itemize}
\item[(C-3)] for every $\varepsilon>0$ and for every Borel set $F\subset E$
  with $\L(F)>0$, there exists $N\in\N$ so that
  \begin{equation*}
   \L\bigl(\bigl\{x\in F:\frac{\L(F\cap A_n(x))}{\L(A_n(x))}\ge 1-\varepsilon
   \bigr\}\bigr)\ge (1-\varepsilon)\L(F)
  \end{equation*}
  for all $n\geq N$, and
\item[(C-4)] for all Borel sets $E_1,E_2\subset E$ with positive
 Lebesgue measure and for every $\varepsilon>0$, there exists $N\in\N$ such
 that
 \begin{equation*}
  \begin{split}
    \iint_{E_1\times E_2}&\iint_{A_n(x)\times A_m(y)} |u-v|^{-s}\;
     d\L(u)d\L(v)d\L(x)d\L(y)\\
    \leq&(1+\varepsilon)\iint_{E_1\times E_2} \L(A_n(x))\L(A_m(y))|x-y|^{-s}\;
     d\L(x)d\L(y)
   \end{split}
 \end{equation*}
 for all $n,m\geq N$.
\end{itemize}
\end{lemma}

\begin{proof}
Property (C-1) is clearly valid. Note that
$\L(F)=\lim_{\delta\to 0}\L(\overline V_\delta(F))$
for all compact sets $F\subset\R^d$ (recall \eqref{dneighbourhood}).
Let $n\in\N$. For every $\delta>0$,  Definition~\ref{sequenceAn}(a) implies
that $A_n(y)\subset \overline V_\delta(A_n(x))$ if
$y$ is close to $x$. Thus $\lim_{y\to x}\L\bigl(A_n(x)\setminus A_n(y)\bigr)=0$
by the continuity of $x\mapsto\L(A_n(x))$. For every $\varepsilon>0$ and
$x\in U$,  Definition~\ref{sequenceAn}(a) guarantees the existence of
$\delta_1,\delta_2>0$ such that
$\L\bigl(\overline V_{\delta_1}(A_n(y))\setminus A_n(y)\bigr)<\varepsilon$ for all
$y\in B(x,\delta_2)$. This, in turn, gives
$\lim_{y\to x}\L\bigl(A_n(y)\setminus A_n(x)\bigr)=0$, implying (C-2).
Property (C-3) follows from Lemma~\ref{prod} and Definition~\ref{sequenceAn}
(ii) and (iii). Finally, (C-4) is given by
Lemma~\ref{intestimate} and Definition~\ref{sequenceAn} (ii) and (iii).
\end{proof}

Now we are ready to prove that the assumption \eqref{e-8'} of
Lemma~\ref{lem-6}
is satisfied for compact sets.

\begin{proposition}\label{assumptionOK}
Let $F\subset E$ be a compact set with $\L(F)>0$. Then for every
$\varepsilon>0$, $\delta>0$
and $n\in\N$, there exist an integer $N>n$ and a Borel measurable set
$\Omega\subset U^\N$ with $\P(\Omega)>1-\delta$ such that
\begin{equation}\label{e-8''}
\int_\Omega\Gamma_s\bigl(F\cap\bigcup_{i=n}^N A_i(x_i)\bigr)\; d\P(\x)
  <\Gamma_s(F)+\varepsilon.
\end{equation}
\end{proposition}

\begin{proof}
Let $\varepsilon>0$, $\delta>0$ and $n\in\N$. Choose
$\mu=\sum_{i=1}^k c_i\L|_{F_i}\in\CP_0(F)$ satisfying
\[
I_s(\mu)<\Gamma_s(F)+\tfrac{\varepsilon}2.
\]
Let $0<\gamma<1$ be sufficiently small (which will be determined later). By
partitioning $F_i$ into smaller Borel sets, if necessary, such that each new
$F_i$ is an approximate level set of the density $h$ with small diameter and
with $\eta(F_i)>0$ (recall that $\L(F)>0$ implies $\eta(F)>0$
by \eqref{e-10}),
we may assume that for all $i=1,\dots,k$ and $m\geq n$,
\begin{equation}\label{e-2.5'}
\sup_{x,y\in F_i}\max\Bigl\{\frac{h(x)}{h(y)},\;\frac{\L(A_m(x))}{\L(A_m(y))},\;
 \frac{I_s(A_m(x))}{I_s(A_m(y))}\Bigr\}\leq 1+\gamma.
\end{equation}
For every $i=1,\dots,k$, fix $z_i\in F_i$. Define for all $m\ge n$,
\[
F_{i,m}:=\bigl\{x\in F_i:\frac{\L(F_i\cap A_m(x))}{\L(A_m(x))}>1-\gamma\bigr\}.
\]
Using the fact that the map $x\mapsto\L(F\cap A_n(x))$ is Borel measurable
(see the proof of Lemma~\ref{prod}), we deduce that $F_{i,m}$ is a Borel set.
By (C-3) and (C-4), there exists an integer $M>n$ such that
\begin{equation}\label{Fimeasure}
\L(F_{i,m})\geq (1-\gamma)\L(F_i)
\end{equation}
for all $i=1,\dots,k$ and $m\geq M$ and, moreover,
\begin{equation}\label{e-2.5}
\begin{split}
 \iint_{F_i\times F_j}&\iint_{A_m(x)\times A_{p}(y)}|u-v|^{-s}\;
   d\L(u)d\L(v)d\L(x)d\L(y)\\
\leq&(1+\gamma)^3\L(A_m(z_i))\L(A_p(z_j))\iint_{F_i\times F_j}|x-y|^{-s}\;
   d\L(x)d\L(y)
 \end{split}
 \end{equation}
for all $i=1,\dots,k$ and $m,p\geq M$.

Applying Lemma~\ref{largesum} with $a_m=g_s(A_m(z_i))$ and
$\omega_m=\chi_{F_{i,m}}g_s(A_m(z_i))$ (recall \eqref{gtdef}),
Definition~\ref{sequenceAn} (iv) together with inequalities
\eqref{bidiffeobound} and \eqref{Fimeasure} imply that we may choose integers
$M_1:=M<M_2<\cdots<M_{k+1}$ recursively such that
\[
\P\bigl(\bigl\{\x\in U^\N:\sum_{m=M_i}^{M_{i+1}-1}\chi_{F_{i,m}}(x_m) g_s(A_m(z_i))
  \geq\gamma^{-1}\bigr\}\bigr)\geq 1-\frac{\gamma}k
\]
for all $i=1,\dots,k$. Let $N:=M_{k+1}$. Define
\begin{equation}\label{Omegai}
\Omega_i:=\bigl\{\x\in U^\N:\sum_{m=M_i}^{M_{i+1}-1}\chi_{F_{i,m}}(x_m) g_s(A_m(z_i))
  \geq\gamma^{-1}\bigr\}
\end{equation}
for all $i=1,\ldots,k$ and set
\begin{equation}\label{Omega}
\Omega:=\bigcap_{i=1}^k \Omega_i.
\end{equation}
Then $\Omega$ is a Borel set with $\P(\Omega)\geq 1-\gamma$.

For all $\x\in\Omega$, we define a finite Borel measure $\mu^\x$ as
\[
\mu^{\x}:=\sum_{i=1}^k\sum_{m\in S_i(\x)} c_{i,m}(\x)\L|_{F_i\cap A_m(x_m)},
\]
where
\begin{align*}
S_i(\x)&:= \{m\in \N:\; M_{i}\leq m<M_{i+1}, \; x_m\in F_{i,m}\} \text{ and}\\
c_{i,m}(\x)&:=c_i\L(F_i)\frac{\L(A_m(z_i))}{I_s(A_m(z_i))}
  \Bigl(\sum_{p\in S_i(\x)} g_s(A_p(z_i))\Bigr)^{-1}.
\end{align*}
Since $F$ and $A_i(x_i)$ are compact the measure $\mu^\x$ is supported on
$F\cap\bigcup_{i=n}^NA_i(x_i)$. Notice that if $x_m\in F_{i,m}$, the inequality
\eqref{e-2.5'} results in
\[
\L(F_i\cap A_m(x_m))\geq (1-\gamma)\L(A_m(x_m))\geq(1-\gamma)(1+\gamma)^{-1}
  \L(A_m(z_i))
\]
which, in turn, yields
\begin{align*}
\|\mu^\x\|&=\sum_{i=1}^k c_i\L(F_i)\sum_{m\in S_i(\x)}\frac{\L(A_m(z_i))
  \L(F_i\cap A_m(x_m))}{I_s(A_m(z_i))}\Bigl(\sum_{p\in S_i(\x)}
  g_s(A_p(z_i))\Bigr)^{-1}\\
&\geq(1-\gamma)(1+\gamma)^{-1}\sum_{i=1}^k c_i\L(F_i)\sum_{m\in S_i(\x)}
  \frac{\L(A_m(z_i))^2}{I_s(A_m(z_i))}\Bigr(\sum_{p\in S_i(\x)}
  g_s(A_p(z_i))\Bigr)^{-1}\\
&=(1-\gamma)(1+\gamma)^{-1}
\end{align*}
where $\| \mu^\x\|$ represents the total mass of $\mu^\x$.
Since $\|\mu^\x\|^{-1}\mu^\x\in\CP_0\bigl(F\cap\bigcup_{i=n}^NA_i(x_i)\bigr)$, we
have
\[
\Gamma_s\bigl(F\cap\bigcup_{i=n}^NA_i(x_i)\bigr)\leq I_s(\|\mu^\x\|^{-1}\mu^\x)
  =\frac{I_s(\mu^\x)}{\|\mu^\x\|^2}\leq\frac{(1+\gamma)^2}{{(1-\gamma)^2}}
  I_s(\mu^\x).
\]

In what follows, we estimate $\int_\Omega I_s(\mu^\x)\; d\P(\x)$. Set
\[
\CS_\Omega:=\bigl\{\bS=(S_i)_{i=1}^k:S_i\subset\{M_i,M_i+1,\ldots,M_{i+1}-1\}
 \text{ and }\sum_{p\in S_i} g_s(A_p(z_i))\geq\gamma^{-1}\bigr\}.
\]
For $\bS\in\CS_\Omega$, define
\begin{equation}
\label{e-ali}
\begin{split}
\pi^{-1}(\bS):=\bigcap_{i=1}^k\{\x\in U^\N:\; &x_m\in F_{i,m}\text{ if }
m\in S_i,\text{ and }
  x_m\in (F_{i,m})^c \\
&\text{if }m\in\{M_i,M_i+1,\ldots,M_{i+1}-1\}\setminus S_i\}.
\end{split}
\end{equation}
Clearly, $\pi^{-1}(\bS)$ is a Borel set.
Observe that $\Omega=\bigcup_{\bS\in\CS_\Omega}\pi^{-1}(\bS)$, where the union is
disjoint. Let
\[
J_s(A,B):=\iint_{A\times B}|x-y|^{-s}\; d\L(x)d\L(y)
\]
for all Lebesgue measurable sets $A,B\subset\R^d$. Consider $\bS\in\CS_\Omega$
and
define $Q_{S_i}:=\sum_{m\in S_i}g_s(A_m(z_i))$ for all $i=1,\dots,k$. Then
\begin{equation}\label{Isstart}
\begin{split}
\int_{\pi^{-1}(\bf S)}I_s(\mu^{\x})\; d\P(\x)
 \le &\sum_{i=1}^k\sum_{j=1}^k\sum_{m\in S_i}\sum_{p\in S_j}\frac{c_ic_j\L(F_i)\L(F_j)
   \L(A_m(z_i))\L(A_p(z_j))}{Q_{S_i}Q_{S_j}I_s(A_m(z_i))I_s(A_p(z_j))}\\
  &\times\int_{\pi^{-1}(\bf S)}J_s(A_m(x_m),A_p(x_p))\; d\P(\x).
\end{split}
\end{equation}
In order to complete the proof of Proposition~\ref{assumptionOK}, we need
two more lemmas.

\begin{lemma}\label{lem-prod}
Let $(Y,\mathcal F,\nu)$ be a probability space, and let
$u\colon Y\times Y\to\R$ and $\tilde u\colon Y\to\R$ be non-negative
measurable functions.
Let $E_1,\ldots,E_N\in\mathcal F$ with $\nu(E_i)>0$ for all
$i=1,\dots,N$. Then we have
\begin{equation*}
\int\Big(\prod_{i=1}^N\chi_{E_i}(y_i)\Big)u(y_1,y_2)\,\prod_{j=1}^N d\nu(y_j)
 =\frac{\prod_{i=1}^N\nu(E_i)}{\nu(E_1)\nu(E_2)}\int_{E_1\times E_2} u(y_1, y_2)
\,d\nu(y_1) d\nu(y_2)
\end{equation*}
and
\begin{equation*}
\int\Big(\prod_{i=1}^N \chi_{E_i}(y_i)\Big)\tilde u(y_1)\,\prod_{j=1}^N d\nu(y_j)
 =\frac{\prod_{i=1}^N \nu(E_i)}{\nu(E_1)} \int_{E_1} \tilde{u}(y_1)\; d\nu(y_1).
\end{equation*}
\end{lemma}
\begin{proof}
The equalities follow from simple calculations.
\end{proof}

We will use the Landau big O notation in the sense that, given
positive functions $g_1,g_2\colon\R\to\R$, the notation
$g_1(\gamma)\le(1+O(\gamma))g_2(\gamma)$ means that there exist
$C,\delta>0$ such
that $g_1(\gamma)\le (1+C\gamma)g_2(\gamma)$ when $0<\gamma<\delta$.

\begin{lemma}\label{lem-2.2}
Let $\bS\in\CS_\Omega$. For all $i,j=1,\dots,k$, $m\in S_i$ and
$p\in S_j$, the following properties hold:
\begin{itemize}
\item[(i)] If $m\neq p$, then
\begin{align*}
\int_{\pi^{-1}(\bf S)}&J_s(A_m(x_m),A_p(x_p))\; d\P(\x)\\
 &\le(1+O(\gamma))
  \frac{\P(\pi^{-1}(\bS))}{\L(F_i)\L(F_j)}\L(A_m(z_i))\L(A_p(z_j))J_s(F_i,F_j),
\end{align*}
\item[(ii)] If $m=p$ (which implies that $i=j$), then
\begin{align*}
\int_{\pi^{-1}(\bf S)}J_s(A_m(x_m),A_p(x_p))\; d\P(\x)\le(1+O(\gamma))
  \P(\pi^{-1}(\bS))I_s(A_m(z_i)).
\end{align*}
\end{itemize}
\end{lemma}

\begin{proof}
We begin by verifying (i). Recall that by \eqref{e-ali},
\[
\chi_{\pi^{-1}({\bf S})}(\x)=\prod_{i=1}^k \Big(\prod_{m\in S_i} \chi_{F_{i,m}}(x_m)\Big)
\cdot \Big(\prod_{n\in S_i^*} \chi_{(F_{i,n})^c}(x_n)\Big)
\]
where $S_i^*:=\{M_i,M_i+1,\ldots,M_{i+1}-1\}\setminus S_i$.  Notice also that
$\eta(F_{i,m}),\eta(F_{j,p})>0$ by \eqref{Fimeasure} and
\eqref{e-10}. Applying Lemma \ref{lem-prod}, we deduce
\begin{equation*}
\begin{aligned}
&\int_{\pi^{-1}(\bf S)}J_s(A_m(x_m),A_p(x_p))\; d\P(\x)\\
 &=\frac{\P(\pi^{-1}(\bS))}{\eta(F_{i,m})\eta(F_{j,p})}\iint_{F_{i,m}\times F_{j,p}}
   J_s(A_m(x),A_p(y))\;d\eta(x)d\eta(y)\\
 &\le\frac{(1+\gamma)^2}{(1-\gamma)^2}\frac{\P(\pi^{-1}(\bS))}{\L(F_i)\L(F_j)}
  \iint_{F_{i,m}\times F_{j,p}}\iint_{A_m(x)\times A_p(y)}|u-v|^{-s}
  &d\L(u)d\L(v)d\L(x)d\L(y)\,\\
&\phantom{e}&(\text{by \eqref{e-2.5'} and \eqref{Fimeasure}})\\
 &\le\frac{(1+\gamma)^5}{(1-\gamma)^2}\frac{\P(\pi^{-1}(\bS))}{\L(F_i)\L(F_j)}
  \L(A_m(z_i))\L(A_p(z_j))J_s(F_i,F_j).&(\text{by \eqref{e-2.5}})
 \end{aligned}
\end{equation*}

To prove (ii), we apply \eqref{e-2.5'} and Lemma \ref{lem-prod} to obtain
\begin{align*}
\int_{\pi^{-1}(\bf S)}&J_s(A_m(x_m),A_p(x_p))\; d\P(\x)=\frac{\P(\pi^{-1}(\bS))}
  {\eta(F_{i,m})}\int_{F_{i,m}}I_s(A_m(x))\; d\eta(x)\\
 &\le(1+\gamma)^2\frac{\P(\pi^{-1}(\bS))}{\L(F_{i,m})}\int_{F_{i,m}}I_s(A_m(z_i))\;
  d\L(x)
 =(1+\gamma)^2\P(\pi^{-1}(\bS))I_s(A_m(z_i)).
\end{align*}
\end{proof}

Now we continue the proof of Proposition~\ref{assumptionOK}. Recalling
\eqref{Isstart} and \eqref{Omegai}, and applying
Lemma~\ref{lem-2.2}, yields
\begin{align*}
\int_{\pi^{-1}(\bf S)}I_s(\mu^{\x})\; d\P(\x)&\le\P(\pi^{-1}(\bS))
   (1+O(\gamma))\Bigl(\sum_{i=1}^k\sum_{j=1}^kc_ic_jJ_s(F_i,F_j)\\
 &\quad+\sum_{i=1}^kc_i^2\L(F_i)^2(Q_{S_i})^{-1}\Bigr)\\
 &\le\P(\pi^{-1}(\bS))(1+O(\gamma))\Bigl(I_s(\mu)
   +\gamma\bigl(\sum_{i=1}^kc_i\L(F_i)\bigl)^2\Bigr)\\
 &\le\P(\pi^{-1}(\bS))(1+O(\gamma))(I_s(\mu)+\gamma).
\end{align*}
Thus, by the choice of $\mu$, we have
\begin{align*}
\int_\Omega I_s(\mu^\x)\; d\P(\x)&=\sum_{\bS\in\CS_\Omega}\int_{\pi^{-1}(\bf S)}
  I_s(\mu^{\x})\; d\P(\x)
\le\P(\Omega)(1+O(\gamma))(I_s(\mu)+\gamma)\\
 &\le(1+O(\gamma))(\Gamma_s(F)+\tfrac{\varepsilon}2+\gamma).
\end{align*}
The claim follows by choosing sufficiently small $\gamma$.
\end{proof}

We complete this section by proving Theorem~\ref{main}.(b).

\begin{proof}[Proof of Theorem~\ref{main}.(b)]
We start by reducing the claim to the setting of
Proposition~\ref{assumptionOK}.
We may assume that $s_0(\bA)>0$. Consider $s<s_0(\bA)$.
Since $G_s(E)=0$ for all
$E\subset\R^d$ with $\L(E)=0$, we may assume that $\L(A_n)>0$ for all
$n\in\N$ by removing sets with $\L(A_n)=0$ from the original sequence
$\bA=(A_n)_{n\in\N}$ if necessary.
Since $I_s(E)\le I_s(F)$ for $E\subset F$ and since
$\mathcal L$ is inner regular, replacing $A_n$ by a suitable subset, we may
assume that $A_n$ is compact for all $n\in\mathbb N$ and
\begin{equation}\label{e-1}
\sum_{n=1}^\infty g_s(A_n)=\infty.
\end{equation}

We proceed by constructing a sequence $(K_n)_{n\in\mathbb N}$ of compact sets
satisfying Definition~\ref{sequenceAn} (i)--(iv) such that $K_n\subset A_n$
for
all $n\in\mathbb N$. Indeed, let $(Q_i)_{i=1}^{m_1}$ be the closed dyadic
cubes with
side length $2^{-1}$ intersecting $\Delta$. Notice that for any Borel set
$E\subset\Delta$, we have $E=\cup_{i=1}^{m_1}E\cap Q_i$ and,
moreover,
there exists $i\in\{1,\dots,m_1\}$ with $\L(E\cap Q_i)\ge\frac 1{m_1}\L(E)$.
Thus,
\[
\sum_{i=1}^{m_1}g_s(E\cap Q_i)=\sum_{i=1}^{m_1}
   \frac{\L(E\cap Q_i)^2}{I_s(E\cap Q_i)}
\ge\sum_{i=1}^{m_1}\frac{\L(E\cap Q_i)^2}{I_s(E)}
\geq \frac{\L(E)^2}{(m_1)^2I_s(E)}=\frac{1}{(m_1)^2}g_s(E).
\]
It follows that
\[
\sum_{i=1}^{m_1}\sum_{j=1}^\infty g_s(A_j\cap Q_i)=\sum_{j=1}^\infty\sum_{i=1}^{m_1}
   g_s(A_j\cap Q_i)
\ge\frac{1}{(m_1)^2}\sum_{j=1}^\infty g_s(A_j)=\infty.
\]
Therefore, there exists
$k_0\in\{1,\dots,m_1\}$ such that
$\sum_{j=1}^\infty g_s(A_j\cap Q_{k_0})=\infty$.
Define $\widetilde{Q}_1:=Q_{k_0}$. We pick integers
$n_1<n_2<\cdots <n_{N_1}$ so that
\[
\L(A_{n_i}\cap\widetilde{Q}_1)>0\text{ for all }i=1,\ldots,N_1
\text{  and  }\sum_{i=1}^{N_1} g_s(A_{n_i}\cap\widetilde{Q}_1)\ge1.
\]
Since $\sum_{j=N_1+1}^\infty g_s(A_j\cap\widetilde{Q}_1)=\infty$, a similar
argument shows that there exist a dyadic cube
$\widetilde{Q}_2\subset\widetilde{Q}_1$ with side
length $2^{-2}$, and positive integers
$n_{N_1+1}<\cdots<n_{N_2}$ such that
\[
\L(A_{n_i}\cap\widetilde{Q}_2)>0\text{ for all }
i=N_{1}+1,\ldots,N_2\text{ and }\sum_{i=N_1+1}^{N_2} g_s(A_{n_i}
\cap\widetilde{Q}_2)\ge1.
\]
Repeat this process inductively. As a
result, we find a decreasing sequence $(\widetilde Q_l)_{l\in\mathbb N}$ of
dyadic cubes, an increasing sequence $(N_l)_{l\in\mathbb N}$ of integers
and an increasing  sequence $(n_l)_{l\in\mathbb N}$ of indices such that for
every $k=0, 1\ldots,$
\[
\L(A_{n_i}\cap\widetilde{Q}_{k+1})>0\text{ for all }
i=N_{k}+1,\ldots,N_{k+1}\text{ and }\sum_{j=N_k+1}^{N_{k+1}}g_s(A_{n_j}\cap
\widetilde Q_{k+1})\ge 1.
\]
Defining
$K_j:=A_{n_j}\cap \widetilde Q_{k+1}$ for every $j=N_k+1,\dots,N_{k+1}$, gives
$\sum_{j=1}^\infty g_s(K_j)=\infty$ and $\lim_{n\to\infty}\diam K_n=0$. Finally,
setting $\{y_0\}:=\bigcap_{k=1}^\infty \widetilde Q_k$, leads to
\[
\lim_{n\to\infty}\dist(y_0,K_n)=0.
\]
Hence, the sequence $(K_n)_{n\in\N}$ satisfies
Definition~\ref{sequenceAn} (i)--(iv).

Since the measure $\sigma$ determining the probability $\P$
(recall the Introduction) is not singular with respect to $\L$, there exist a
compact set $E\subset U$ such that $\sigma|_E(E)>0$, $\sigma|_E\ll \L$ and
\eqref{e-10} is satisfied with
$h:=\frac{d\sigma|_E}{d\L}$.
Let $f\colon U\times V\to\R^d$ be as in the introduction. For all
$(x,y)\in U\times V$, let
$T_x:=f(x,\cdot)^{-1}$ and $T^y:=f(\cdot,y)^{-1}$. Then for all $x\in U$,
the set
$f(U,y_0)\cap f(x,V)$ is non-empty (it always contains the point
$f(x,y_0)$) and open,
$y_0\in V_x:=T_x\bigl(f(U,y_0)\cap f(x,V)\bigr)$ and
$x\in U_x:=T^{y_0}\bigl(f(U,y_0)\cap f(x,V)\bigr)$. Thus the map
$W_x:V_x\to U_x$ defined by $W_x(v):=T^{y_0}(f(x,v))$ is a diffeomorphism with
$W_x(y_0)=x$ and
\[
\Vert DW_x\Vert\,,\,\Vert (DW_x)^{-1}\Vert\le (C_u)^2
\]
where $C_u$ is as in inequality \eqref{diffeobound}.
Clearly, the derivative of the map
$x\mapsto W_x(v)$ has the same bounds. Let $O$ be an open and bounded set
such that $E\subset O\subset\overline O\subset U$. Consider
$0<r_0<\min\{\dist(y_0,V^c),(C_u)^{-2}\dist(\overline O,U^c)\}$. Then
$B(y_0,r_0)\subset V_x$ for all $x\in O$. Thus $W:O\times B(y_0,r_0)\to\R^d$
satisfies Definition~\ref{sequenceAn} (a)--(c). Ignoring a finite number of
sets $K_n$, if necessary, we may assume that $K_n\subset B(y_0,r_0)$ for all
$n\in\N$. We conclude that all the conditions in Definition~\ref{sequenceAn}
are fulfilled.

As a result of the fact that $T^{y_0}$ is a diffeomorphism, we
conclude that
\[
\dimH\bigl(\limsup_{n\to\infty}f(x_n,K_n)\bigr)
\ge\dimH\bigl(\limsup_{n\to\infty}W(x_n,A_n)\bigr)
\]
for all $\x\in U^\N$. Here we have an inequality instead of an equality since
for $x_n\in U\setminus O$ it may happen that $K_n\not\subset V_{x_n}$.
Finally, the claim follows by combining Proposition~\ref{assumptionOK},
Lemma~\ref{lem-6} and Proposition~\ref{pro-1.5}.
\end{proof}

\section{Packing dimension of random covering sets}\label{packingdimension}

In this section, we prove Theorem~\ref{main}.(d). For the purpose of studying packing dimensions of random covering sets,
we set
\[
N_\ell^*(E):=\#\{Q\in\CQ_\ell:\L(Q\cap E)>0\}
\]
for all $E\subset\R^d$ and $\ell\in\N$.
Here the symbol $\#$ stands for the cardinality and $\CQ_\ell$ is as in
\eqref{e-CQ}. We begin with a result concerning a lower bound for
packing dimensions of intersections of decreasing sequences of compact sets.

\begin{lemma}\label{lem-pack}
Let $(E_n)_{n\in\N}$ be a decreasing sequence of compact subsets of $\R^d$
with positive Lebesgue measure. Let $s>0$. Assume that there exists a sequence
$(\ell_n)_{n\in\N}$ of natural numbers such that
\begin{equation}\label{e-3a}
N_{\ell_{n+1}}^*(Q\cap E_{n+1})\geq 2^{\ell_{n+1}s}
\end{equation}
for all $n\in\N$ and for all $Q\in\CQ_{\ell_n}$ with $\L(Q\cap E_n)>0$.
Then $\dimp\bigl(\bigcap_{n=1}^\infty E_n\bigr)\geq s$.
\end{lemma}

\begin{proof}
For $n\in \N$, set
\[
F_n:=E_n\cap\Bigl(\bigcup_{\substack{Q\in\CQ_{\ell_n}\\ \L(Q\cap E_n)>0}}\overline Q\Bigr),
\]
and let $F_\infty:=\bigcap_{n=1}^\infty F_n$. Clearly, $F_n\subset E_n$ is compact
and
$\L(F_n)=\L(E_n)$. Hence, $N_{\ell_n}^*(Q\cap F_n)=N_{\ell_n}^*(Q\cap E_n)$ for all
$n\in\N$ and $Q\in\bigcup_{k=1}^\infty\CQ_k$. In particular, we have
\begin{equation}\label{e-3a'}
N_{\ell_{n+1}}^*(Q\cap F_{n+1})\geq 2^{\ell_{n+1}s}
\end{equation}
for all $n\in\N$ and $Q\in\CQ_{\ell_n}$ with $\L(Q\cap F_n)>0$.
Denoting by $\dimbu$ the upper box counting dimension, we will show that
\begin{equation}\label{e-3b}
\dimbu(V\cap F_\infty)\geq s
\end{equation}
for all open sets $V$ with $V\cap F_\infty\neq\emptyset$.
By the Baire category theorem, \eqref{e-3b} implies that $\dimp(F_\infty)\geq s$
(see for example \cite[Proposition 3.6 and Corollary 3.9]{Fal03}) and,
therefore,
$\dimp\bigl(\bigcap_{n=1}^\infty E_n\bigr)\geq s$, as desired.

To prove \eqref{e-3b}, let $V$ be an open set so that
$V\cap F_\infty\neq\emptyset$. Then there exist $n\in\N$ and $Q\in\CQ_{\ell_n}$
such that $2Q\subset V$ and $Q\cap F_n\neq\emptyset$, where $2Q$ stands for the
union of all elements $Q'\in\CQ_{\ell_n}$ with
$\overline{Q'}\cap\overline Q\neq\emptyset$. By the definition of $F_n$, there
is $Q^*\in\CQ_{\ell_n}$ such that
$\overline{Q^*}\cap\overline Q\neq\emptyset$ and
$\L(Q^*\cap F_n)>0$. Since $Q^*\subset 2Q\subset V$, replacing $Q$ by $Q^*$, if
necessary, we may assume that $\L(Q\cap F_n)>0$. Using \eqref{e-3a'}
recursively, leads to
\begin{equation}\label{e-QQ}
N_{\ell_m}^*(Q\cap F_{m})\geq 2^{\ell_{m}s}
\end{equation}
for all $m>n$. Furthermore, we claim that for every $m>n$,
\begin{equation}\label{e-QQ'}
\#\{Q'\in\CQ_{\ell_m}:\overline{Q'}\cap Q\cap F_\infty\}
\geq N_{\ell_m}^*(Q\cap F_{m})\geq 2^{\ell_ms},
\end{equation}
from which we conclude that $\dimbu(Q\cap F_\infty)\geq s$ and, therefore,
$\dimbu(V\cap F_\infty)\geq s$.
To prove \eqref{e-QQ'}, it follows from \eqref{e-QQ} that it is enough
to show
that $\overline{Q'}\cap F_\infty\neq\emptyset$ for all
$Q'\in\CQ_{\ell_m}$ with $\L(Q'\cap F_m)>0$.
For this purpose, consider $Q'\in\CQ_{\ell_m}$ with  $\L(Q'\cap F_m)>0$. By
\eqref{e-3a'}, there exists $Q'_1\in\CQ_{\ell_{m+1}}$ such that
$Q'_1\subset Q'$ and $\L(Q'_1\cap F_{m+1})>0$.  Using this fact recursively,
we
see that for every $p\in\N$, there exists $Q'_{p}\in\CQ_{\ell_{m+p}}$ such that
$Q'_p\subset Q'_{p-1}$ and $\L(Q'_p\cap F_{m+p})>0$. Hence, we have
$\L(Q'\cap F_{m+p})>0$ for all $p\in\N$, which implies that
$\overline{Q'}\cap F_\infty\neq\emptyset$. This completes the proof of
the lemma.
\end{proof}

Before applying the above result to estimate packing dimensions of
random covering sets, we prove several lemmas.

\begin{lemma}\label{lem-3.0}
For all $A\in\B(\R^d)$ and $\ell\in\Z$, we have
$N_\ell^*(A)\geq 2^{\ell d}\L(A)$.
\end{lemma}

\begin{proof}
The claim follows directly from a simple volume argument.
\end{proof}

Let $U\subset\R^d$ be open, and let $(A_n(x))_{n\in\N}$ be a sequence of
compact-set-valued functions defined on $U$ satisfying the conditions (C-1) and
(C-2) from Section~\ref{regularenergy}. Let $\eta\in\CP(U)$ and set
$\P:=\prod_{i=1}^\infty\eta$.

\begin{lemma}\label{lem-3.1}
Let $E\in\B(\R^d)$ with $0<\L(E)<\infty$, and let $\ell\in\Z$.
Then the mapping
\[
(x_i)_{i=1}^n\mapsto N_\ell^*\bigl(E\cap\bigcup_{i=1}^n A_i(x_i)\bigr)
\]
is lower semi-continuous on $U^n$ for all $n\in\N$. Moreover, the mapping
\[
\x\mapsto N_\ell^*\bigl(E\cap\bigcup_{i=1}^\infty A_i(x_i)\bigr)
\]
is Borel measurable on $U^\N$.
\end{lemma}

\begin{proof}
It suffices to prove the first part of the lemma; the second part follows
directly from the first one and the following easily-checked identity:
\begin{equation}\label{e-3.1'}
N_\ell^*\bigl(E\cap\bigcup_{i=1}^\infty A_i(x_i)\bigr)
=\lim_{n\to\infty}N_\ell^*\bigl(E\cap\bigcup_{i=1}^nA_i(x_i)\bigr).
\end{equation}

Let $(x_i)_{i=1}^n\in U^n$ and write
$k:=N_\ell^*\bigl(E\cap\bigcup_{i=1}^nA_i(x_i)\bigr)$ for short. Then there are
$k$ different elements in $\CQ_{\ell}$, say $Q_1,\ldots,Q_k$, such that
$\L\bigl(Q_j\cap\bigcup_{i=1}^n A_i(x_i)\bigr)>0$ for all $j=1,\ldots,k$. It
follows from
(C-2) that
$\L\bigl(\bigcup_{i=1}^n A_i(x_i)\setminus\bigcup_{j=1}^n A_j(y_j)\bigr)$
is close to $0$
when $(y_i)_{i=1}^n\in U^n$ is close to $(x_i)_{i=1}^n$ and,
therefore, when $(y_i)_{i=1}^n$ is in a small neighbourhood of
$(x_i)_{i=1}^n$, we have $\L\bigl(Q_j\cap\bigcup_{i=1}^n A_i(y_i)\bigr)>0$ for all
$j=1,\ldots,k$. Hence,
$N_\ell^*\bigl(E\cap\bigcup_{i=1}^nA_i(y_i)\bigr)\geq k$, concluding the proof
of lower semi-continuity.
\end{proof}

The following result may be regarded as an analogy of Proposition~\ref{pro-1.5}.

\begin{proposition}\label{pro-3.1}
Let $E\subset U$ be compact with $\eta(E)>0$. Suppose that $\eta|_E\ll\L$.
Moreover, assume that for every $\ell,n\in\N$ and for every compact sets
$F\subset E$
with $\eta(F)>0$,
\begin{equation}\label{e-3.5}
N^*_\ell\bigl(F\cap\bigcup_{i=n}^\infty A_i(x_i)\bigr)=N_\ell^*(F)
\end{equation}
for $\P$-almost all $\x\in U^\N$. Then
\[
\dimp\bigl(\limsup_{n\to\infty}A_n(x_n)\bigr)=d
\]
for $\P$-almost all $\x\in U^\N$.
\end{proposition}

\begin{proof}
Replacing $E$ by a compact subset, if necessary, we may assume that
$0<\frac{d\eta|_E}{d\L}(x)<\infty$ for all $x\in E$. Thus,
for all $F\subset E$, we have
\begin{equation}\label{Lequiveta}
\L(F)>0\text{ if and only if }\eta(F)>0.
\end{equation}
Let $\varepsilon,\delta>0$. It suffices to verify that
\begin{equation}\label{e-3d}
\P\bigl(\bigl\{\x\in U^\N:\dimp\bigl(\limsup_{n\to\infty}A_n(x_n)\bigr)
\geq d-\delta\big\}\bigr)\geq 1-\varepsilon.
\end{equation}
For this purpose, we are going to construct a Borel set
$\Omega\subset U^\N$ with $\P(\Omega)>1-\varepsilon$, and two sequences
$(\ell_k)_{k\in\N}$ and $(m_k)_{k\in\N}$ of natural numbers such that
for all $\x\in\Omega$, $k\in\N$ and $Q\in\CQ_{\ell_k}$, we have
\begin{equation}\label{e-Q}
N_{\ell_{k+1}}^*\bigl(Q\cap E\cap\bigcap_{j=1}^{k+1}\bigcup_{i=m_j+1}^{m_{j+1}}A_i(x_i)
\bigr)\geq 2^{\ell_{k+1}(d-\delta)}
\end{equation}
provided that
$\L\bigl(Q\cap E\cap\bigcap_{j=1}^k\bigcup_{i=m_j+1}^{m_{j+1}}A_i(x_i)\bigr)>0$.
By Lemma~\ref{lem-pack}, this implies that
\[
\dimp\bigl(\limsup_{n\to\infty}A_n(x_n)\bigr)\geq d-\delta
\]
for all $\x\in\Omega$, from which \eqref{e-3d} follows.

Now we present our construction. Set $\ell_1:=1$ and $m_1:=1$. Notice that
\[
\gamma_1:=\min\{\L(Q\cap E):Q\in\CQ_{\ell_1}\text{ and }\L(Q\cap E)>0\}>0.
\]
Choosing a large integer $\ell_2>\ell_1$ so that $2^{-\ell_2\delta}<\gamma_1$, it
follows from Lemma~\ref{lem-3.0} that
\begin{equation}\label{e-N1}
N_{\ell_2}^*(Q\cap E)\geq 2^{\ell_2 d}\gamma_1>2^{\ell_2(d-\delta)}
\end{equation}
for all $Q\in\CQ_{\ell_1}$ with $\L(Q\cap E)>0$.
Hence, by \eqref{e-3.5}, for $\P$-almost all $\x\in U^\N$ and for
all $Q\in\CQ_{\ell_1}$ with $\L(Q\cap E)>0$, we have
\[
N_{\ell_2}^*\bigl(Q\cap E\cap\bigcup_{i=m_1+1}^{\infty}A_i(x_i)\bigr)
=N_{\ell_2}^*(Q\cap E)>2^{\ell_2(d-\delta)},
\]
where we used \eqref{Lequiveta} and the fact that
$N_\ell^*(\overline Q\cap A)=N_\ell^*(Q\cap A)$ for
all $\ell\in\N$, $Q\in\CQ_\ell$ and $A\subset\R^d$. By
\eqref{e-3.1'} and Lemma~\ref{lem-3.1}, we find a large integer $m_2>m_1$
and a Borel set $\Lambda_2\subset U^{m_2}$
with $\eta^{m_2}(\Lambda_2)>1-\frac{\varepsilon}2$ such that for all
$(x_1,\ldots,x_{m_2})\in\Lambda_2$
\[
N_{\ell_2}^*\bigl(Q\cap E\cap\bigcup_{i=m_1+1}^{m_2}A_i(x_i)\bigr)>
 2^{\ell_2(d-\delta)}
\]
for all $Q\in\CQ_{\ell_1}$ with $\L(Q\cap E)>0$.
Define a mapping $\tau_2:\Lambda_2\to (0,\infty)$ by
\begin{align*}
\tau_2(x_1,\ldots,x_{m_2}):=\min\bigl\{\L\bigl(Q\cap E\cap &\bigl(
   \bigcup_{i=m_1+1}^{m_2} A_i(x_i)\bigr)\bigr):Q\in\CQ_{\ell_2}\\
 &\text{ with }\L\bigl(Q\cap E\cap\bigl(\bigcup_{i=m_1+1}^{m_2}A_i(x_i)\bigl)\bigr)
   >0\bigr\}.
\end{align*}
By (C-2), the function $\tau_2$ is continuous and, hence, Borel measurable on
$\Lambda_2$. Since $\tau_2(x_1,\dots,x_{m_2})>0$ for all
$(x_1,\dots,x_{m_2})\in\Lambda_2$, there exist $\gamma_2>0$
and a Borel set $\Lambda_2'\subset\Lambda_2$ such that
\[
\eta^{m_2}(\Lambda_2')>\eta^{m_2}(\Lambda_2)-\frac{\varepsilon}6
>1-\frac{2\varepsilon}3
\]
and
\[
\tau_2(x_1,\ldots,x_{m_2})\geq\gamma_2
\]
for all $(x_1,\ldots,x_{m_2})\in\Lambda'_2$.
Choose $\ell_3>\ell_2$ so that $2^{-\ell_3\delta}<\gamma_2$.
Lemma~\ref{lem-3.0} implies that for all $(x_1,\ldots, x_{m_2})\in\Lambda'_2$
and $Q\in\CQ_{\ell_3}$,
\begin{align*}
N_{\ell_3}^*\bigl(Q\cap E\cap\bigl(\bigcup_{i=m_1+1}^{m_2}A_i(x_i)\bigr)\bigr)
\geq 2^{\ell_3d}\gamma_2>2^{\ell_3(d-\delta)}
\end{align*}
provided that $\L\bigl(Q\cap E\cap(\bigcup_{i=m_1+1}^{m_2}A_i(x_i))\bigr)>0$.
Again, by \eqref{e-3.5}, we find $m_3>m_2$ and a Borel set
$\Lambda_3\subset\Lambda_2'\times\prod_{i=m_2+1}^{m_3}U\subset\Lambda_2
\times\prod_{i=m_2+1}^{m_3}U
\subset U^{m_3}$ such that
\[
\eta^{m_3}(\Lambda_3)>\eta^{m_2}(\Lambda_2')-\frac{\varepsilon}{12}
 >1-\frac{3\varepsilon}4,
\]
and, moreover, for all $(x_1,\ldots, x_{m_3})\in\Lambda_{m_3}$ and
$Q\in\CQ_{\ell_3}$,
\begin{align*}
N_{\ell_3}^*\bigl(Q\cap E\cap\bigcap_{j=1}^2\bigl(\bigcup_{i=m_j+1}^{m_{j+1}}A_i(x_i)
   \bigr)\bigr)>2^{\ell_3(d-\delta)}
 \end{align*}
provided that $\L\bigl(Q\cap E\cap(\bigcup_{i=m_1}^{m_2}A_i(x_i))\bigr)>0$.

Continuing the above process, we construct recursively two increasing
sequences $(\ell_k)_{k\in\N}$ and $(m_k)_{k\in\N}$ of integers
and a sequence $(\Lambda_k)_{k\in\mathbb N}$ of
Borel sets such that $\Lambda_k\subset U^{m_k}$,
$\Lambda_{k+1}\subset\Lambda_k\times\prod_{i=m_k+1}^{m_{k+1}}U$,
$\eta^{m_k}(\Lambda_k)>1-\frac{(2k-1)\varepsilon}{2k}$ and inequality
\eqref{e-Q} holds for all $(x_1,\ldots,x_{m_{k+1}})\in\Lambda_{k+1}$. Setting
$\Omega:=\bigcap_{k=1}^\infty(\Lambda_k\times\prod_{i=m_k+1}^\infty U)$, gives
$\P(\Omega)=\lim_{k\to\infty}\eta^{m_k}(\Lambda_k)\geq 1-\varepsilon$ and,
moreover, \eqref{e-Q} holds for all $\x\in\Omega$.
This completes the proof.
\end{proof}

Now we are ready to prove our main result on the packing dimension of random
covering sets.

\begin{theorem}\label{mainpacking}
Let $E\subset\R^d$ be compact with $\eta(E)>0$. Suppose that $\eta|_E$ is
equivalent with $\L|_E$. Let $(A_n(x))_{n\in\N}$ be a sequence of
compact-set-valued functions defined on $U$  satisfying the conditions (C-1)
and (C-2) from Section~\ref{regularenergy}. In addition, suppose that for all
compact sets  $F\subset E$ with $\L(F)>0$
\begin{equation}\label{e-pa}
\sum_{n=1}^\infty\eta\bigl(\{x\in F:\L(F\cap A_n(x))>0\}\bigr)=\infty.
\end{equation}
Then for $\P$-almost all $\x\in U^\N$,
\[
\dimp\bigl(\limsup_{n\to\infty}A_n(x_n)\bigr)=d.
\]
\end{theorem}

\begin{proof}
Let $\ell\in\N$, and let $F\subset E$ be compact with $\L(F)>0$. By
Proposition~\ref{pro-3.1}, it is sufficient to prove that for
all $n\in\N$,
\begin{equation}\label{e-equ}
N_\ell^*\bigl(F\cap\bigcup_{i=n}^\infty A_i(x_i)\bigr)=N_\ell^*(F)
\end{equation}
for $\P$-almost all $\x\in U^\N$.
Note that \eqref{e-equ} is equivalent to the statement that for all
$Q\in\CQ_\ell$ with $\L(Q\cap F)>0$,
\begin{equation}\label{e-end}
\L\bigl(Q\cap F\cap\bigcup_{i=n}^\infty A_i(x_i)\bigr)>0\text{ for }
\P\text{-almost all }\x\in U^\N.
\end{equation}
  Fix $Q\in\CQ_\ell$ with $\L(Q\cap F)>0$.  For all $k\in\mathbb N$,
we consider the independent events
\[
E_k:=\{x_k\in Q\cap F:\L(Q\cap F\cap A_k(x_k))>0\}.
\]
Replacing $F$ by $\overline Q\cap F$ in \eqref{e-pa},
we have $\sum_{k=1}^\infty \eta(E_k)=\infty$.
Applying the second Borel-Cantelli lemma, yields
\eqref{e-end}.
 \end{proof}

We complete this section by proving Theorem~\ref{main}.(d).

\begin{proof}[Proof of Theorem~\ref{main}.(d)]
Recall from the introduction that $A_n(x_n)=f(x_n,A_n)$. Since $\L$ is inner
regular, we may assume that the sets $A_n$ are compact with
$\L(A_n)>0$ and properties (C-1) and (C-2) are satisfied. Let $F\subset U$ be
a compact set with $\L(F)>0$ such that $\sigma|_E$ is equivalent with
$\L|_F$.
As in the proof of Theorem~\ref{main}.(b), we
may replace $f(x_n,A_n)$ by $W(x_n,A_n)$. Then
\eqref{e-pa} follows from Lemma~\ref{prod}. Hence,
Theorem~\ref{mainpacking} implies the claim.
\end{proof}

\section{Proof of Corollary~\ref{corollary} and examples}\label{examples}

The aim of this section is to verify Corollary~\ref{corollary} and to discuss
the sharpness of our results. We begin by
proving Corollary~\ref{corollary} as
a consequence of Theorem~\ref{main}.

\begin{proof}[Proof of Corollary~\ref{corollary}]
Observe that for all sequences $(E_n)_{n\in\mathbb N}$ and $(F_n)_{n\in\mathbb N}$
of subsets of ${\pmb N}$, we have
$\limsup_{n\to\infty}(E_n\cup F_n)=(\limsup_{n\to\infty} E_n)\cup
 (\limsup_{n\to\infty} F_n)$.
Therefore, covering ${\pmb M}$ and ${\pmb N}$ by a finite number of coordinate
charts, we may assume that $f:{\pmb K}\times V\to\R^d$ satisfies the
assumptions in Theorem~\ref{main} with the exception that $f(\cdot,y)$ is
not necessarily injective.
Recall that as long as \eqref{diffeobound} is valid,
the dependence of $f$ on the first coordinate plays no role when
proving that $t_0(\bA)$ is an upper bound for the dimension and
$s_0(\bA)=t_0(\bA)$ (see Section~\ref{upperbound}). In Sections
\ref{regularenergy}--\ref{packingdimension}, where the lower bounds are
proven, we restrict our
considerations into a bounded set where $\sigma$ is absolutely continuous with
respect to the Lebesgue measure. We deduce that, covering ${\pmb K}$ by
a finite number
of coordinate charts, we may assume that $f:U\times V\to\R^d$ is as in
Theorem~\ref{main} and, therefore, Corollary~\ref{corollary} follows from
Theorem~\ref{main}.
\end{proof}

We continue by constructing examples that demonstrate the sharpness of our
results.
We begin with showing that the lower bound proven by Persson (see
\eqref{Persson}) is not always sharp.

\begin{example}\label{G>g}
Let $Q_1,Q_2\subset\R^d$ be disjoint open cubes
with side lengths $r_1$ and $r_2$, respectively. Let $0<\rho<1$.
Divide $Q_2$ into $2^{nd}$ subcubes $Q_2^j$ and set
$F_{Q_2}:=\bigcup_{j=1}^{2^{nd}}\rho Q_2^j$, where $\rho Q$ is the concentric cube
with $Q$ having side length $\rho$ times that of $Q$. Define
$A:=Q_1\cup F_{Q_2}$. Using the change of variables $x'=r_ix$ for $i=1,2$, one
easily sees that $I_t(Q_1)=C_1r_1^{-t}\mathcal L(Q_1)^2$ and
$I_t({ F}_{Q_2})\le C_2r_2^{-t}\mathcal L({ F}_{Q_2})^2$, where $C_1$ and $C_2$ are
constants depending only on $d$ and $t$. Choosing sufficiently small
$\rho>0$, guarantees that
$\mathcal L(A)<2\mathcal L(Q_1)$ which, in turn, implies that
\[
\frac{g_t(A)}{g_t(F_{Q_2})}\le\frac{4\L(Q_1)^2}{I_t(Q_1)}\frac{I_t(F_{Q_2})}
  {\L(F_{Q_2})^2}\le C\Bigl(\frac{r_1}{r_2}\Bigr)^t,
\]
where $C$ is a constant. Hence,
$G_t(A)\ge g_t(F_{Q_2})\ge C^{-1}(\frac{r_2}{r_1})^tg_t(A)$.
Since the ratio $\frac{r_2}{r_1}$ can be chosen arbitrarily large, we
conclude that even for open sets and for $\sigma:=\L$, the lower bound given
for $\dimH\bE(\x,\bA)$ in \cite{Pe} by means of $g_t$ may be strictly
smaller than
the quantity $s_0(\bA)$ in Theorem~\ref{main} (see \eqref{defs0}).
\end{example}

Next we give an example which shows that if we replace the
assumption that every $A_n$ has positive Lebesgue density by
a weaker assumption
that $\L(A_n\cap B(x,r))>0$ for all $n\in\N$, $x\in A_n$ and $r>0$,
Corollary~\ref{H=G} is not valid  and $\dimH\bE(\x,\bA)$ can be
almost surely strictly smaller than $t_0(\bA)$.

\begin{example}\label{posdensityupper}
Let $\sigma:=\L$ on $\mathbb T^2$ and set $\P:=\prod_{i=1}^\infty \sigma$.
Define
$f(x,y):\mathbb T^2\times\mathbb T^2\to\mathbb T^2$ by $f(x,y)=x+y$
for all $(x,y)\in\mathbb T^2\times\mathbb T^2$.
For every $n\in\N$, let $E_n:=[0,1]\times\{0\}\subset\mathbb T^2$ and
$F_n:=\bigcup_{i=1}^\infty B(y_i,2^{-n-i})\subset\mathbb T^2$ where the
centres $y_i$ are dense in $E_n$. Set $A_n:=E_n\cup F_n$ and write
$\bA:=(A_n)_{n\in\N}$, $\bE:=(E_n)_{n\in\N}$ and
$\bF:=(F_n)_{n\in\N}$. We deduce that
$\L(A_n\cap B(y,r))>0$ for all $r>0$ and $y\in A_n$ but
\[
\liminf_{r\to0}\frac{\mathcal L(A_n\cap B(y,r))}{\mathcal L(B(y,r))}=0
\]
for $\mathcal H^1$-almost all $y\in E_n\setminus F_n$, which follows by applying the Lebesgue density  theorem for $\mathcal H^1|_{E_n}$ and noting that $L(A_n\cap B(y,r))\le 2r \mathcal H^1(B(y,r)\cap E_n\cap F_n)$. Recall that
\[
\limsup_{n\to\infty}(x_n+A_n)=\limsup_{n\to\infty}(x_n+E_n)\cup
  \limsup_{n\to\infty}(x_n+F_n).
\]
Now $\sum_{n=1}^\infty\mathcal H_\infty^1(E_n)=\infty$ and
$\sum_{n=1}^\infty\mathcal H_\infty^t(F_n)<\infty$ for all $t>0$.
Thus $t_0(\bA)=1$ and $t_0(\bF)=0$. By
Corollary~\ref{corollary}, we have
$\dimH\bigl(\limsup_{n\to\infty}(x_n+F_n)\bigr)=0$ for all
$\x\in(\mathbb T^2)^\N$. Furthermore,
$\limsup_{n\to\infty}(x_n+E_n)=\emptyset$ $\P$-almost surely, since
\[
\P\bigl((x_n+E_n)\cap(x_m+E_m)\ne\emptyset\text{ for some }n,m\in\N
\text{ with }
  n\ne m\bigr)=0.
\]
We conclude that $\dimH\bigl(\limsup_{n\to\infty}(x_n+A_n)\bigr)=0<1=t_0(\bA)$
$\P$-almost surely.
Observe that $s_0(\bA)=s_0(\bF)=0$.
\end{example}

Next we construct an example illustrating that if the generating sets $A_n$
do
not have positive Lebesgue density it is possible that
$\dimH\bE(\x,\bA)>s_0(\bA)$
almost surely. For this purpose, we recall the following notation from
\cite{Fal92}.

\begin{definition}\label{Falconer}
For all $0<s\le d$, let
\begin{align*}
\mathcal G^s(\R^d):=\{F\subset\R^d:\,&F\text{ is a }G_\delta\text{-set such that }
  \dimH(\bigcap_{i=1}^\infty f_i(F))\ge s\text{ for all}\\
  &\text{similarities }f_i\colon\R^d\to\R^d, i\in\N\}.
\end{align*}
We say that the sets in the class $\mathcal G^s(\R^d)$ have
large intersection property.
\end{definition}

In \cite[Theorem A]{Fal92}, Falconer showed that $\mathcal G^s(\R^d)$ is the
maximal class of $G_\delta$-sets of Hausdorff dimension at least $s$ which is
closed under countable intersections and similarities. Moreover, in
\cite[Theorem B]{Fal92}, he gave several equivalent ways to define the class
$\mathcal G^s(\R^d)$, one of them being
\begin{equation}\label{netmeasure}
F\in\mathcal G^s(\R^d)\iff\mathcal M_\infty^s(F\cap Q)=\mathcal M_\infty^s(Q)
  \text{ for all dyadic cubes }Q
\end{equation}
where $\mathcal M_\infty^s$ is the $s$-dimensional net content defined as in
\eqref{Hauscontent} with covering sets being dyadic cubes.
Definition \eqref{netmeasure}
was extended by Bugeaud \cite{Bu} and Durand \cite{Du07}
for general gauge functions and open subsets of $\R^d$.

\begin{example}\label{posdensity}
Let $\sigma:=\L$ on $\mathbb T^d$, $\P:=\prod_{i=1}^\infty\sigma$ and define
$f\colon\mathbb T^d\times\mathbb T^d\to\mathbb T^d$ by $f(x,y)=x+y$
for all $(x,y)\in\mathbb T^d\times\mathbb T^d$. Consider
$0<s<t<d$ and choose a sequence $(\widetilde B_i)_{i\in\N}$ of open
balls such that $\dimH\bE(\x,(\widetilde B_i)_{i\in\N})=t$ for
$\P$-almost all $\x\in(\mathbb T^d)^\N$. Fix such
a typical covering set and denote it by $F$. Assume that
$(B_i)_{i\in\N}$ is a sequence of open balls such that
$\dimH\bE(\x,(B_i)_{i\in\N})=s$ for
$\P$-almost all $\x\in(\mathbb T^d)^\N$.
Let $(r_i)_{i\in\N}$ be a decreasing sequence of positive real numbers
which tends to 0 so slowly that
$\bE(\x,(B(0,\frac{r_i}2)_{i\in\N}))=\mathbb T^d$ for $\P$-almost all
$\x\in(\mathbb T^d)^\N$.
(for existence of such
$(r_i)_{i\in\N}$ see \cite{Ka90}). Viewing
$\mathbb T^d=[-\frac 12,\frac 12[^d\subset\R^d$, we define $A_i:=r_iF\cup B_i$
for all $i\in\mathbb N$ and set $\bA:=(A_i)_{i\in\N}$. The fact that
$\dimH F<d$
implies that { $\L(F)=0$ and,  hence,} $G_t(A_i)=G_t(B_i)$ for all $i\in\N$, giving
\[
s=s_0(\bA)=\sup\{0\le u\le d:\sum_{i=1}^\infty G_u(A_i)=\infty\}.
\]
By \cite[Theorem 2]{Du}, we have
$F\in\mathcal G^t(]-\frac 12,\frac 12[^d)$. Let
$\widetilde F$ be the lift of $F$ to $\R^d$ by a covering map. { We claim that $\widetilde F\in\mathcal G^t(\R^d)$. Indeed, to prove this claim, by
\cite[Lemma 10]{Du07}, it is
enough to show that the equality in \eqref{netmeasure} (in which $F$ is replaced by $\widetilde F$) holds for all dyadic cubes $Q$ with small diameter. This is the case,  since $F\in\mathcal G^t(]-\frac 12,\frac 12[^d)$ and $\widetilde F$ is the lift of $F$. }  Since $\mathcal G^t(\R^d)$ is closed under
countable intersections and similarities by \cite[Theorem A]{Fal92}, we obtain
$\widetilde H(\x):=\bigcap_{i=n}^\infty(x_i+r_i\widetilde F)\in\mathcal G^t(\R^d)$
for all $\x\in\mathbb T^d$ and, thus,
$H(\x):=\widetilde H(\x)\cap ]-\frac 12,\frac 12[^d
  \in\mathcal G^t(]-\frac 12,\frac 12[^d)$
by \cite[Proposition 1]{Du07}. Since
$\bE(\x,(B(0,\frac{r_i}2)_{i\in\N}))=\mathbb T^d$ for $\P$-almost all
$\x\in\mathbb T^d$, every point of $\mathbb T^d$ belongs to $B(x_i,\frac{r_i}2)$
for infinitely many $i\in\N$. Using the fact that the sequence $(r_i)_{i\in\N}$
tends to zero, we conclude that
$]-\frac 12,\frac 12[^d\subset\bigcup_{i=n}^\infty B(x_i,\frac{r_i}2)
  \cap ]-\frac 12,\frac 12[^d$
for all $n\in\N$. Combining this with the fact
$\widetilde H(\x)\cap B(x_i,\frac{r_i}2)\cap ]-\frac 12,\frac 12[^d
  \subset x_i+r_iF$
for all $i\ge n$, leads to
$H(\x)\subset\bigcup_{i=n}^\infty(x_i+r_iF)$ for $\P$-almost all
$\x\in\mathbb T^d$. By \cite[Proposition 1]{Du07}, every $G_\delta$-set
containing a subset in $\mathcal G^t(]-\frac 12,\frac 12[^d)$ belongs to
$\mathcal G^t(]-\frac 12,\frac 12[^d)$.
Thus, almost surely
\[
\dimH\bigl(\bigcap_{n=1}^\infty\bigcup_{i=n}^\infty(x_i+r_iF)\bigr)\ge t.
\]
Hence, $\dimH\bE(\x,(r_iF)_{i\in\N})\,,\,\dimH\bE(\x,\bA)\ge t>s_0(\bA)$
for $\P$-almost all $\x\in(\mathbb T^d)^\N$. In particular,
$\dimH\bE(\x,\bA)>s_0(\bA)$ for $\P$-almost all $\x\in(\mathbb T^d)^\N$.
\end{example}

Finally, we give examples which show that Theorem~\ref{main} fails if the
distribution $\sigma$ is singular with respect to the Lebesgue measure.

\begin{example}\label{distribution}
(a) Let $f(x,y)$ be as in Example~\ref{posdensity} and let
$\sigma:=\delta_{x_0}$ for some $x_0\in\mathbb T^d$.
Set $\P=:\prod_{i=1}^\infty\sigma$. Defining
$A_n:=B(0,n^{-\frac 1d})\setminus\{0\}$, we obtain
$\sum_{n=1}^\infty g_d(A_n)=\infty$ and $s_0(\bA)=t_0(\bA)=d$. However,
$\limsup_{n\to\infty}(x_n+A_n)=\emptyset$ $\P$-almost surely.
Thus, Theorem~\ref{main} is not valid.

(b) Let $s<d$ and let $C$ be the regular $2^d$-corner Cantor set on
$\mathbb T^d$ with $\dimH C=\dimp C=s$. Set $\sigma:=\mathcal H^s|_C$ and
assume that everything
else is as in example (a). Then $\bE(\x,\bA)\subset C$ almost surely.
In particular,
\[
\dimH\bE(\x,\bA)\le\dimp\bE(\x,\bA)\le s<s_0(\bA)=t_0(\bA)=d
\]
$\P$-almost surely. Hence, for every $s<d$ there exists
a measure $\sigma$ with ${\dimH\sigma}=s$ for which Theorem~\ref{main} fails.
\end{example}

\section{Further generalizations and  remarks}\label{finalremarks}

\subsection{A weak large intersection property of random covering sets} In  \cite{Du, Pe}, it is proved that, when $\bA=(A_n)_{n\in \N}$ is a sequence of open balls or general open sets on ${\Bbb T}^d$ so that $\sum_{n=1}^\infty g_s(A_n)=\infty$ for some $0<s\leq d$, then
almost surely the random covering set $\bE(\x,\bA)$ has the large intersection
property in the sense that $\bE(\x,\bA)\in\mathcal G^s$ (cf.
Definition~\ref{Falconer}).  We remark that this result also holds under a weaker condition  that  $\sum_{n=1}^\infty G_s(A_n)=\infty$, because one may find open subsets $B_n$ of $A_n$ so that $\sum_{n=1}^\infty g_s(B_n)=\infty$,   according to the following  easily checked fact:
 $$G_s(A)=\sup\{g_s(B): \; B\subset A, \; B\mbox{ is open}\}$$
 whenever $A$ is open.  We  emphasise that, in the above investigation, the assumption of  $A_n$ being open is essential and cannot be dropped, for otherwise     $\bE(\x,\bA)$ may not be  a $G_\delta$-set.

 Nevertheless, in the general setting that the sets in $\bA$ are Lebesgue measurable, we obtain the following  weak large intersection property of random covering sets.

 \begin{theorem}
 Assuming that $\bA$ is a sequence of  Lebesgue measurable sets,  we have under the conditions of Theorem~\ref{main} that 
  \[
\dimH\bigl(\bigcap_{j=1}^\infty\bE(\x^j,\bA)\bigr)\ge s_0(\bA)
\]
for $(\prod_{j=1}^\infty\P)$-almost all $(\x^j)_{j\in\N}\in\prod_{j=1}^\infty U^\N$.
 \end{theorem}
\begin{proof}
This can be verified by modifying the proof of Proposition~\ref{pro-1.5}
in the
following manner: Let $\varphi:\N\to\N\times\N$ be a bijection obtained using
the diagonal method. Repeat the construction of Proposition~\ref{pro-1.5} such
that the $n$-th construction step is done using the variable $\x^{\varphi(n)_1}$,
where $\varphi(n)_1$ is the first coordinate of $\varphi(n)$.
This leads to the conclusion
\[
\dimH\bigl(\bigcap_{n=1}^\infty\bigcup_{i=N_{n-1}+1}^{N_n}A_i(x_i^{\varphi(n)_1})\bigr)
  \ge s
\]
(cf. \eqref{midresult}), which implies the desired result.
\end{proof}

\subsection{Hausdorff measure of random covering sets} Let $d\in \N$. Denote by $\mathcal G$ the collection of
functions $h\colon [0,\infty[\to[0,\infty[$ such that
$h$ is increasing, positive near $0$, $\lim_{r\to 0}h(r)=h(0)=0$, and  $h(r)r^{-d}$ is
decreasing. Any element of $\mathcal G$ is called a gauge function. For $F\subset \R^d$ and $h\in \mathcal G$, we use $\mathcal H^h(F)$ and $\mathcal H^h_\infty(F)$ to denote the Hausdorff measure and Hausdorff content of $F$ with respect to the gauge function $h$ (cf. \cite{Car, Rog}).  For instance,  $\mathcal H^h_\infty(F)$ is defined by replacing
$({\rm diam} F_n)^s$ by $h({\rm diam} F_n)$ in the definition \eqref{Hauscontent}.

In \cite{Du}, Durand  studied the Hausdorff measures of random covering sets on ${\Bbb T}^d$ when $\bA$ is a sequence of balls of the form $A_n=B(0, r_n)$.
Using the mass transference principle established in  \cite{BV}, he
showed that for any $h\in \mathcal G$ with $\lim_{r\to 0}h(r)r^{-d}=\infty$, almost surely
$$
\mathcal H^h(\bE(\x,\bA))=\left\{\begin{array}{ll}
\infty &\mbox{ if} \quad \sum_{n=1}^\infty h(r_n)=\infty,\\
0 & \mbox{ otherwise}.
\end{array}
\right.
$$
However,  this  approach does not extend to the general case when the sets in $\bA$ are not ball-like, since the mass transference principle may fail in such situation.

To deal with the general case, let us introduce some notation. For   a Lebesgue measurable set $F\subset\R^d$ with $\L(F)>0$ and $h\in \mathcal G$, we define the $h$-energy of $F$ by
\[
I_h(F):=\iint_{F\times F}h(\vert x-y\vert)^{-1}\,d\mathcal L(x)d\mathcal L(y).
\]
Set $g_h(F):=\L(F)^2I_h(F)^{-1}$ and use $g_h$ to define $G_h(F)$ as in
\eqref{defG}. Following the argument in the proof of Lemma~{\ref{HgeG} with routine changes, we can show that
\begin{equation}
\label{HgeG1}
\mathcal H^h_\infty(F)\geq G_h(F).
\end{equation}

As a generalisation of  Theorem~\ref{main}, we have the following result on the Hausdorff measures of general random covering sets.
\begin{theorem}
\label{Gauge-extension}
Let $h\in \mathcal G$. Under the assumptions of Theorem~\ref{main}, we have
\begin{itemize}
\item[(i)] $\sum_{n=1}^\infty \mathcal H^h_\infty(A_n)<\infty \Longrightarrow \mathcal H^h(\bE(\x,\bA))=0$.
\item[(ii)] $\sum_{n=1}^\infty G_h(A_n)=\infty\implies\mathcal H^h(\bE(\x,\bA))=\infty$ for $\P$-almost all $\x\in U^\N$,  provided that $I_h(B(0,R))<\infty$ for all
$R>0$ and $A_n$ are Lebesgue measurable.
\item[(iii)] Assume that  $r\mapsto h(r)r^{-d+\e}$ is decreasing for some $\varepsilon>0$ and,  moreover, assume that $\tilde h\in\mathcal G$ is such that  the inequality $\tilde h(r)\le h(r)^{1+\delta}$ is
valid for some $\delta>0$ and all $r>0$. Then  $$\sum_{n=1}^\infty G_h(A_n)<\infty\implies\sum_{n=1}^\infty \mathcal H^{\tilde{h}}_\infty(A_n)<\infty,$$ provided that $A_n$ are  Lebesgue measurable  with positive Lebesgue density.
\end{itemize}
\end{theorem}
\begin{proof}
Statement (i) follows from  a routine modification of the proof of Lemma~\ref{H_inftyupperbound}. Statement (ii) follows from the proof of Theorem~\ref{main}.(b) with slight modifications. Indeed,
in the proof of Theorem~\ref{main}.(b), the only
place where the fact that the kernel is $\vert x\vert^{-s}$ is needed
is inequality \eqref{bigdistance} (see the proof of
Lemma~\ref{intestimate}). To extend that inequality associated to $h$,  it is enough to have that
\begin{equation}\label{epsilonsmooth}
h(r)\le\bigl(1+O(\e)\bigr)h((1-\e)r)\text{ for all }0<r<2R.
\end{equation}
Note that  $h$ is doubling in the sense that $h(2r)<ch(r)$ for some constant $c>1$, which follows from the fact that $h(r)r^{-d}$ is decreasing.  Hence,  the gauge function $\tilde h$ obtained from
$h$ as the linear interpolation of $h$ at points $2^{-n}$, $n\in\N$, is
equivalent with $h$ and satisfies \eqref{epsilonsmooth}. Now
Proposition~\ref{pro-1.5} implies that $\mathcal H^h(\bE(\x,\bA))>0$
$\P$-almost surely. It is not difficult to see that if
$\sum_{n=1}^\infty G_h(A_n)=\infty$ there exists a gauge function $h'$ such that
$\lim_{r\to 0}h'(r)h(r)^{-1}=0$ and $\sum_{n=1}^\infty G_{h'}(A_n)=\infty$.
Therefore, $\mathcal H^{h'}(\bE(\x,\bA))>0$
which implies $\mathcal H^h(\bE(\x,\bA))=\infty$.

The proof of  (iii) is essentially identical to that of  Lemma~\ref{HleG}.
Observe that one may assume that $\mathcal H^{\tilde h}(B(0,R))>0$ for some
$R>0$ since otherwise the claim is trivial. The assumption that
$h(r)r^{-d+\e}$ is decreasing is needed at the end of the proof of
Lemma~\ref{lemma-5} when the term (II) is estimated.
(Recall that Lemma~\ref{lemma-5} is needed in the proof of
Proposition~\ref{thm:energy_lemma}). Observe that heuristically
$\mathcal H^{\tilde h}(B(0,R))>0$ means that $\tilde h(r)$ should be larger
than $r^d$ for small $r>0$ and, therefore, $h(r)$ should be larger than
$r^{\frac d{1+\delta}}$ for small $r>0$.
\end{proof}

\begin{remark}
{\rm One may expect that for some $R>0$ there exists a constant $C>0$ such that
for all Lebesgue measurable sets $F\subset B(0,R)$,
\begin{equation}\label{InHG}
\mathcal H_\infty^h(F)\le CG_h(F).
\end{equation}
If so,  the condition $\sum_{n=1}^\infty G_h(A_n)=\infty$ in Theorem~\ref{Gauge-extension}.(ii) can be replaced by
$$\sum_{n=1}^\infty {\mathcal H}_\infty^h(A_n)=\infty.$$
 However, \eqref{InHG} does not hold  for general
doubling gauge functions even in the case where $F$ is a ball. Indeed, let
$h(r)=r^d(\log r)^2$ for all $0<r<r_0$ where $r_0$ is chosen such that $h$ is
increasing. A straightforward calculation implies that $I_h(B(x,r))$ is
comparable to $(r^d|\log r|)^{-1}$. Applying \cite[Theorem 1.15]{Mat} to
product measures, making a discrete approximation and using the fact that the
sum $\sum_{i=1}^n a_i^2$ is minimised for the uniform probability vector
$(a_1,\dots,a_n)$, it is not difficult to see that $g_h(B(x,r))$ is comparable
to $G_h(B(x,r))$. Therefore, $G_h(B(x,r))$ is comparable to $h(r)|\log r|^{-1}$
while $\mathcal H_\infty^h(B(x,r))$ is comparable to $h(r)$.
}
\end{remark}

\begin{remark}
{\rm Here  we indicate how $G_h(F)$ can be calculated for some concrete examples.
Assume that $F=B(x,r)$. It follows immediately from the definition that
$G_h(F)\le h(2r)$. If $h(r)r^{-d+\e}$ is decreasing for some $\e>0$ (thus $h$ is
doubling), one easily sees that $I_h(F)\le C r^{2d}h(r)^{-1}$ for some constant
$C>0$. Therefore, $G_h(F)$ is comparable to $h(r)$. Another easily calculable
example is when $F$ is a rectangle (or parallelepiped in higher dimensions) with side lengths
$a\ge b$. Then $G_s(F)$ is comparable to $a^s$ for $0<s<1$ and to $ab^{s-1}$ for
$1<s<2$.}
\end{remark}

\begin{remark}
{\rm Basing on the above remark, one can verify that \eqref{InHG} holds in the following particular cases: (i) $F$ is a ball and $h$ is a gauge function so that $r\mapsto h(r)r^{-d+\e}$ is decreasing for some $\epsilon>0$; (ii) $F$ is a rectangle, and $h(r)=r^s$ for some non-integer $s\in (0,2)$.
}
\end{remark}

\subsection{A question on the measurability of level sets of random covering sets.}
It is a natural question whether  $\dim\bE(\x,\bA)$ takes a constant value almost surely in the general setting that $\bA$ is a sequence of Lebesgue measurable sets, where $\dim$ is either the Hausdorff, packing
or box counting dimension.
It is obvious that $\dim\bE(\x,\bA)$ does not depend on a
finite number of coordinates $x_i$. Therefore,
$$F_s:=\{\x\in U^\N : \dim\bE(\x,\bA)=s\}$$ is a tail event for every
$0\le s\le d$,  provided that $F_s$ is measurable. In this case, the
Kolmogorov's zero-one law would imply  that $\x\mapsto\dim\bE(\x,\bA)$
is almost surely a constant. Theorem~\ref{main} gives the value of this
constant under further assumptions on $\bA$.

 Using the results of Dellacherie \cite{Del} and Mattila and Mauldin
\cite{MM}, it is easy to see that $F_s$
is measurable with respect to the $\sigma$-algebra
generated by analytic sets provided that the sets $A_n$ are analytic for all
$n\in\N$ (for details see \cite{JJKLSX}). For Lebesgue measurable generating
sets $(A_n)_{n\in \N}$, we do not know whether the sets $F_s$ are measurable
or not.

\end{document}